\documentclass{amsart}
 \UseRawInputEncoding 
\usepackage{amssymb}
\usepackage{graphicx}
\usepackage{fancyhdr}
\usepackage{cite}
\usepackage[backref=page]{hyperref}
\usepackage{xcolor}

\newtheorem{theorem}{Theorem}[section]
\newtheorem{lemma}[theorem]{Lemma}
\newtheorem{proposition}[theorem]{Proposition}

\theoremstyle{definition}

\theoremstyle{remark}
\newtheorem{remark}[theorem]{Remark}

\numberwithin{equation}{section}

\begin{document}

\title{Uniform Complex Time Heat Kernel Estimates Without Gaussian Bounds}

%    Information for first author
%\author{*}
%    Address of record for the research reported here
%\address{}
%    Current address
%\curraddr{}
%\email{}
%    \thanks will become a 1st page footnote.
%\thanks{The first author was supported in part by NSF Grant \#000000.}

%    Information for second author

\author{Shiliang Zhao}
\address{Department of Mathematics, Sichuan University, Chengdu, Sichuan 610064, P.R.China}
\email{zhaoshiliang@scu.edu.cn}

\author{Quan Zheng}
\address{School of Mathematics and Statistics, Huazhong University of Science and Technology, Wuhan, Hubei, 430074, P.R.China}
\email{qzheng@hust.edu.cn}

%\thanks{The authors are supported by the National Natural Science Foundation of China under Grant No. 11901407 and No.11801188 respectively.}

%    General info

%\date{January 1, 2001 and, in revised form, June 22, 2001.}

%\dedicatory{This paper is dedicated to our advisors.}


\begin{abstract}
The aim of the paper is twofold. First we study the uniform complex time heat kernel estimates of $e^{-z(-\Delta)^{\frac{\alpha}{2}}}$ for $\alpha>0, z\in \mathbb{C}^+$. When $\frac{\alpha}{2}$ is not an integer, generally the heat kernel doest not have the Gaussian upper bounds for real time. Thus the Phragm\'en-Lindel\"of methods(for example \cite{CS}) fail to give the uniform complex time heat kernel estimates. Instead, we overcome this difficulty by giving the asymptotic estimates  for $P(z, x)$ as $z$ tending to the imaginary axis. Then we prove the uniform complex time heat kernel estimates. Secondly,  we study the uniform complex time estimates of the analytic semigroup generated by $H=(-\Delta)^{\frac{\alpha}{2}}+V$ where $V$ belongs to higher order Kato class.
\end{abstract}

\pagestyle{fancy}
\fancyhead{}
\fancyhead[CE]{\footnotesize{SHILIANG ZHAO AND QUAN ZHENG}}
\fancyhead[CO]{\footnotesize{UNIFORM COMPLEX TIME HEAT KERNEL ESTIMATES}}
\fancyhead[RO]{\footnotesize{\thepage}}
\fancyhead[LE]{\footnotesize{\thepage}}
\renewcommand{\headrulewidth}{0pt}
\maketitle

\section{Introduction}
Let $e^{-z(-\Delta)^{\frac{\alpha}{2}}}$  be the analytic semigroup generated by $(-\Delta)^{\frac{\alpha}{2}}$ where $\Delta$ is the Laplace operator on $\mathbb{R}^n$ and $\alpha>0$,  $z\in \mathbb{C}^+$ with $\mathbb{C}^+=\{z\in \mathbb{C} | \Re z>0 \}$. Denote by $P(z, \cdot )$ the convolution kernel of $e^{-z(-\Delta)^{\frac{\alpha}{2}}}$ on $L^2(\mathbb{R}^n)$. In fact, by the Fourier transform we have
\begin{equation}\label{heatkernel}
  P(z, x)= c_n \int_{\mathbb{R}^n} e^{ix\cdot \xi} e^{-z |\xi|^\alpha}  d\xi,   \hspace{1cm} \forall x\in \mathbb{R}^n, z\in \mathbb{C}^+,
\end{equation}
where $c_n$ is a constant determined by the dimension. Recently, the fractional Laplace operator has been extensively studied due to its wide applications in nonlinear optics, plasma physics and other areas. See for example \cite{CS, CHKL, DZF, GH, HS, HWZD, HYZ, L2002} and references therein.

In this paper, first we focus on the uniform  estimates of the heat kernel $P(z, x)$ for $z\in \mathbb{C}^+$. Now we recall some known facts about the heat kernel.

When $z=\Re z$ is real, the estimates for $P(\Re z, x)$ are well known. Indeed, there exists constant $C>0$ such that(\cite{BG, HNS, K, MYZ})
\begin{equation}\label{realtime}
  |P(\Re z, x)| \le C \Re z^{-\frac{n}{\alpha}} \wedge \frac{\Re z}{|x|^{n+\alpha}}, \hspace{0.3cm} \forall \alpha, z=\Re z>0, x\in \mathbb{R}^n.
\end{equation}
Throughout this paper, for two functions $f, g$,  set $f\wedge g=  \min \{ f, g \}$.

Moreover, when $\alpha$ are even numbers, the upper bounds can be improved into the sub-Gaussian type upper bounds in  the following sense,
\begin{equation}\label{subGauss1}
  |P(\Re z, x)| \le C_1 \Re z^{-\frac{n}{\alpha}} \exp \left\{ -C_2 \frac{|x|^{\frac{\alpha}{\alpha-1}}}{\Re z^{\frac{1}{\alpha-1}}} \right\}, \hspace{1cm} \forall z=\Re z>0, x\in \mathbb{R}^n,
\end{equation}
for some positive constants $C_1, C_2>0$. See, for example \cite{BD, DDY, HWZD}.

One important way to deduce the uniform complex time heat kernel estimates from the real time heat kernel estimates is by the Phragm\'en-Lindel\"of theorems. Davies in \cite{Davies95} introduced this method to obtain the uniform complex time heat kernel estimates from the Gaussian upper bounds for the real time. Further, Carron et al in \cite{CCO} proved the uniform complex time estimates for heat kernel satisfying the sub-Gaussian upper bounds for real time. In particular, by \cite[Proposition 4.1]{CCO} and  \eqref{subGauss1}, there  exist positive constants $C_1, C_2>0$ such that
\begin{equation}\label{subGauss2}
  |P(z, x)| \le C_1 \Re z^{-\frac{n}{\alpha}}  \exp \left\{ -C_2 \frac{|x|^{\frac{\alpha}{\alpha-1}}}{ |z|^{\frac{1}{\alpha-1}}} \cos \theta, \right\},  ~~~\forall z\in \mathbb{C}^+, x\in \mathbb{R}^n,
\end{equation}
where $\theta= \arg z$ and $\alpha$ are even numbers. For more results concerning the Phragm\'en-Lindel\"of methods and their applications, we refer the readers to \cite{CS2007, Davies95,Davies89, Gri09, Ouh05} and references therein.

However, when $\alpha>0$ is not an even number, the sub-Gaussian estimates do not hold in general and hence the Phragm\'{e}n-Lindel\"{o}f methods in \cite{CS2007, Davies95} fail to give the uniform complex time heat kernel estimates. On the other hand, by simple calculations, there exists $C>0$  such that
\begin{equation}\label{Trivial Estimate}
  |P(z, x)| \le c_n \int_{\mathbb{R}^n} e^{-\Re z |\xi|^\alpha} d\xi \le C \Re z^{-\frac{n}{\alpha}}, \hspace{0.5cm} \forall \alpha>0, z\in \mathbb{C}^+, x\in \mathbb{R}^n.
\end{equation}
To the best of our knowledge, we can not find the uniform complex time estimates in the literature for general $\alpha>0$ except the trivial estimates \eqref{Trivial Estimate}, even though the estimates for $P(z, x)$ are well known when $z$ are real numbers or pure imaginary numbers.

To get the desired results without the Gaussian upper bounds, we investigate the asymptotic behavior of $P(z, x)$ as $|x|\rightarrow 0$ and $|x|\rightarrow \infty$ uniformly for $z$ satisfying $0<\omega \le |\arg z|<\frac{\pi}{2}$. Then our first results are as follows.
\begin{theorem}
Let $\alpha>0$ and $P(z, x)$ be defined by \eqref{heatkernel}.

(1) When $0<\alpha<1$, there exist constants $C_1, C_2>0$ such that for all $ z\in \mathbb{C}^+, x\in \mathbb{R}^n$,  we have,
\begin{equation}\label{theorem2.1}
  |P(z, x)| \le C_1  \left[|z|^{-\frac{n}{\alpha}}+ \frac{|z|^{ \frac{n}{2(1-\alpha)} }  }{ |x|^{n\frac{1-\frac{\alpha}{2}}{1-\alpha}} } \exp \left( -C_2 \frac{|x|^{\frac{\alpha}{\alpha-1}}}{|z|^{\frac{1}{\alpha-1}}} \cos \theta \right)  \right]  \wedge \frac{|z|}{|x|^{n+\alpha}},
\end{equation}
where $\theta=\arg z$.

(2) When $\alpha>1$, there exist constants $C'_1, C'_2>0$ such that for all $z\in \mathbb{C}^+, x\in \mathbb{R}^n$, we have,
\begin{equation}\label{theorem2.2}
  |P(z, x)| \le C_1' |z|^{-\frac{n}{\alpha}} \wedge \left[ \frac{|z|}{|x|^{n+\alpha}} + \frac{|z|^{ \frac{n}{2(1-\alpha)} }  }{ |x|^{n\frac{1-\frac{\alpha}{2}}{1-\alpha}} } \exp \left( -C'_2 \frac{|x|^{\frac{\alpha}{\alpha-1}}}{|z|^{\frac{1}{\alpha-1}}} \cos \theta \right)   \right]£¬
\end{equation}
where $\theta=\arg z$.
\end{theorem}
\begin{remark}
Compared with \eqref{subGauss2}, \eqref{Trivial Estimate}, the estimates are new. When $\alpha$ is an even number, the right hand side of \eqref{subGauss2} tends to infinity as $|\theta|\rightarrow \frac{\pi}{2}$.  However, according to \cite[Proposition 5.1]{M}, the upper bounds in \eqref{theorem2.2} stay true even for $|\theta|=\frac{\pi}{2}$. Moreover, when $\theta=0$, the estimates \eqref{theorem2.1}, \eqref{theorem2.2} correspond with \eqref{realtime}.
%By \eqref{subGauss2}, we know $P(z, y)$ has sub-Gaussian estimates for $\alpha=2m$ , $m\in \mathbb{N}^*$. Thus when $\frac{\alpha}{2}$ is integer, our estimates are not sharp as $|x|\rightarrow \infty$.  However, comparison  between \eqref{subGauss2} and our results \eqref{theorem2.2} shows that the upper bounds in \eqref{subGauss2} can  not be controlled by which  in \eqref{theorem2.2} neither.  Indeed, the arguments in \cite{CCO} work under very general assumptions. Thus it seems reasonable to expect better results in $\mathbb{R}^n$, when $\frac{\alpha}{2}$ is an integer. We refer the readers to \cite{CS2007} for more details.
\end{remark}
Next we consider the heat kernel of $e^{-z((-\Delta)^{\frac{\alpha}{2}} + V  )}$ with $V$ belonging to the higher order Kato class $K_{\alpha}(\mathbb{R}^n)$.
Recall that, for each $\alpha>0$, a real valued measurable function $V(x)$ on $\mathbb{R}^n$ is said to lie in $K_\alpha(\mathbb{R}^n)$ if
$$ \lim_{\delta\rightarrow 0} \sup_{x\in \mathbb{R}^n} \int_{|x-y|<\delta} w_\alpha (x-y) |V(y)| dy=0, \hspace{1cm} \text{for}~~~~ 0<\alpha\le n, $$
and
$$ \sup_{x\in \mathbb{R}^n} \int_{|x-y|<1} |V(y)| dy <\infty, \hspace{1cm} \text{for} ~~~~\alpha>n, $$
where
$$ w_\alpha(x)= \begin{cases} |x|^{\alpha-n}, & \text{if}\quad 0<\alpha<n, \\  \ln |x|^{-n}, & \text{if}\quad \alpha=n.  \end{cases}  $$
Set $I(t, x)= t^{-\frac{n}{\alpha}} \wedge \frac{t}{|x|^{n+\alpha}}$ and denote the integral kernel of  $e^{-z((-\Delta)^{\frac{\alpha}{2}} + V  )}$ by $K(z, x, y)$. Then our results concerning $K(z, x, y)$ are as follows.
\begin{theorem}
Let $\alpha>0$ and $V\in K_\alpha(\mathbb{R}^n)$.

(1) When $0<\alpha<1$, then for any $0<\varepsilon \ll 1,$ there exists constant $C>0$ and $\mu_{\varepsilon, V}$ depending on $V, \varepsilon,$ such that
\begin{equation}\label{theorem3.1}
  |K(z,x,y)|\le C e^{\mu_{\varepsilon, V}|z|} (\cos \theta)^{-\frac{n}{\alpha}+\frac{n}{2}} I(|z|, x-y), \hspace{1cm} \forall z\in \mathbb{C}^+, x, y \in \mathbb{R}^n.
\end{equation}

(2) When $\alpha>1$, then for any $0<\varepsilon \ll 1,$ there exists constant $C'>0$ and $\mu'_{\varepsilon, V}$ depending on $V, \varepsilon,$ such that
\begin{equation}\label{theorem3.2}
  |K(z,x,y)|\le C' e^{\mu'_{\varepsilon, V}|z|} (\cos \theta)^{-\frac{n}{2}-\alpha+1} I( |z|, x-y), \hspace{1cm} \forall z\in \mathbb{C}^+, x, y \in \mathbb{R}^n.
\end{equation}
\end{theorem}

The paper is organized as follows: In section 2, we will show the asymptotic behavior of $P(z, x)$ as $|x|\rightarrow 0$ and $|x|\rightarrow \infty$ uniformly for $z$ satisfying $0<\omega \le |\arg z|<\frac{\pi}{2}$. The proof relies heavily on  properties of Bessel functions and we mainly apply the integration by parts as well as  stationary phase methods.  The calculation however is  complicate. Section 3 is devoted  to Theorem 1.1, Theorem 1.3. We will apply the heat kernel estimates in Theorem 1.1 and some global characterizations of $K_\alpha(\mathbb{R}^n)$ to show Theorem 1.3. In the appendix ,we gather some basic properties of Bessel functions.

Note that the constants $\delta, c, C, C_k, C'_k $ for $k \in \mathbb{N}$ may change from line to line.

\section{Uniform Asymptotic Behavior of $P(z, x)$}
In this section, we denote by $e^{i\theta}=z$ for simplicity.

Recall that, when $z=i\Im z$ is pure imaginary number, the estimates for $P(i\Im z, x)$ are well known. As we shall see, the behaviors of $P(i\Im z, x)$ are quite different from that of the real time heat kernel.  To state the results, we make some reduction. By scaling property, we obtain
$$ P(z, x)= |z|^{-\frac{n}{\alpha}} P(e^{i \theta}, \frac{x}{|z|^{\frac{1}{\alpha}}}), \hspace{1.2cm} \forall |z|\neq 0, x\in \mathbb{R}^n,   $$
where $\theta=\arg z$. Moreover, since $P(e^{-i\theta}, -y)=\overline{P(e^{i\theta}, y)}$,  it is sufficient to consider $P(e^{i\theta}, y)$ for $0\le \theta < \frac{\pi}{2}$, $y\in \mathbb{R}^n.$

First of all,  we have
\begin{align*}
 P(i, y)  &= c_n \int_{\mathbb{R}^n} e^{iy\cdot \xi} e^{-i |\xi|^\alpha}  d\xi \\
   & = c_n \int_{\mathbb{R}^n} e^{iy\cdot \xi} e^{-i |\xi|^\alpha} \varphi(|\xi|)  d\xi +c_n \int_{\mathbb{R}^n} e^{iy\cdot \xi} e^{-i |\xi|^\alpha} (1-\varphi(|\xi|))  d\xi \\
   & \triangleq P_1(i, y) + P_2(i, y),
\end{align*}
where $\varphi(t)$ is a smooth cutoff function which equals 1 for  $0\le t\le \frac{1}{2}$ and 0 for $t\ge1$.

For $P_1(i, y)$, there exists constant $C>0$ such that
$$ |P_1(i, y)| \le C (1+|y|^2)^{-\frac{\alpha+n}{2}} \hspace{1.2cm} \forall \alpha>0, ~~  y\in \mathbb{R}^n. $$
Note that the above estimates are essentially known in various literature. For completeness, we still give a proof below. See for example the proof of \eqref{theorem1.1}.

For $P_2(i, y)$, the asymptotic behaviors vary according to different $\alpha$ and $y$. By Miyachi(\cite[Proposition 5.1]{M} ), Wainger(\cite[p.41-52]{W}), \cite{HHZ}, the properties can be summarized as follows:

When $0<\alpha<1$, $P_2(i, y)$ is smooth in $\mathbb{R}^n\setminus \{0\}$ and for $N>0$
$$ P_2(i, y)= O(|y|^{-N})  \hspace{1cm} \text{as} ~ |y|\rightarrow + \infty. $$
Moreover, we have
$$ P_2(i, y)= C_1|y|^{-n\frac{1-\frac{\alpha}{2}}{1-\alpha}} \exp ( C_2 i |y|^{\frac{\alpha}{\alpha-1}}) + o\left( |y|^{-n\frac{1-\frac{\alpha}{2}}{1-\alpha}}  \right) + E(y), ~~~~ \text{as} ~ y\rightarrow 0 ,$$
where $C_2=\alpha^{\frac{\alpha}{1-\alpha}}(\alpha-1)$, $C_1$ is a constant determined by $\alpha, n$ and $E(y)$ is a smooth function.

When $\alpha>1$, $P_2(i, y)$ is smooth in $\mathbb{R}^n$ and we have
$$ P_2(i, y)= C'_1|y|^{-n\frac{1-\frac{\alpha}{2}}{1-\alpha}} \exp ( C_2 i |y|^{\frac{\alpha}{\alpha-1}}) + o\left( |y|^{-n\frac{1-\frac{\alpha}{2}}{1-\alpha}}  \right) , ~~~~ \text{as} ~ |y|\rightarrow +\infty ,$$
where $C_2=\alpha^{\frac{\alpha}{1-\alpha}}(\alpha-1)$ and $C'_1$ is a constant determined by $\alpha, n$. Then the asymptotic behaviors of $P(i, y)$ will totally be determined by $P_1(i, y),  P_2(i, y)$.

As we have seen, the asymptotic  behaviors of $P(e^{i\theta}, y)$ are quite different between $\theta=0$ and $\theta=\frac{\pi}{2}$ as $|y|\rightarrow 0$ and $|y|\rightarrow \infty$.  Moreover, in the case that $\frac{\alpha}{2}$ is integer, the upper bounds in \eqref{subGauss2} tends to infinity as $\theta \rightarrow \frac{\pi}{2}$. They fail to give the upper bounds for $P(i, y)$. Thus it is natural to ask the question that how $P(e^{i\theta}, y)$ changes as $\theta\rightarrow \frac{\pi}{2}$.

Now we give an example which is heuristic for our problems. Consider
$$ I(z)= \int_0^1  e^{-zt} t^{m-1} dt,  $$
where $z\in \mathbb{C}^+$ and $m\ge 2$ is an integer. Integration by parts gives
$$ I(z)= (m-1)!\left(z^{-m}-e^{-z}\sum_{k=0}^{m-1} \frac{z^{k-m}}{k!}\right) = (m-1)!(z^{-m}- e^{-z}z^{-1} ) + E(z). $$
It follows that $I(z)$ has different behaviors between $z=is$ and $z=s$ as $s\rightarrow +\infty$. However, the main contribution to $I(z, m)$ as $|z|\rightarrow + \infty$ is $z^{-m}- e^{-z}z^{-1} $ uniformly for $\Re z\ge0$.

As we have shown in the example, to determine the asymptotic behavior of $P(e^{i\theta}, y)$ uniformly for $0<  \omega \le |\theta|<\frac{\pi}{2}$, we need to find a balance between the two cases: $\theta=0, \theta=\frac{\pi}{2}$. Set
\begin{align*}
 P(e^{i\theta}, y)  &= c_n \int_{\mathbb{R}^n} e^{iy\cdot \xi} e^{-e^{i\theta} |\xi|^\alpha}  d\xi \\
   & = c_n \int_{\mathbb{R}^n} e^{iy\cdot \xi} e^{-e^{i\theta} |\xi|^\alpha} \varphi(|\xi|)  d\xi +c_n \int_{\mathbb{R}^n} e^{iy\cdot \xi} e^{-e^{i\theta} |\xi|^\alpha} (1-\varphi(|\xi|))  d\xi \\
   & \triangleq P_1(e^{i\theta}, y) + P_2(e^{i\theta}, y),
\end{align*}
where $\varphi(t)$ is a smooth cutoff function which equals 1 for  $0\le t\le \frac{1}{2}$ and 0 for $t\ge1$. Then our first result is as follows.
\begin{proposition}
Let $\alpha>0$,  and $ P_1(e^{i\theta}, y), P_2(e^{i\theta}, y)$ be defined as above. Then the following hold:

(1) There exists positive constant $C>0$ such that
\begin{equation}\label{theorem1.1}
  |P_1(e^{i\theta}, y)| \le C (1+|y|^2)^{-\frac{\alpha+n}{2}}, \hspace{1.2cm} \forall \alpha>0, |\theta|\le \frac{\pi}{2},  y\in \mathbb{R}^n.
\end{equation}
(2) When $0<\alpha<1$, for $0< \omega \le |\theta|<\frac{\pi}{2}$ there exists constant $C_{\theta, 1}$ such that
\begin{equation}\label{theorem1.2}
  P_2(e^{i\theta}, y)= C_{\theta, 1}|y|^{-n\frac{1-\frac{\alpha}{2}}{1-\alpha}} \exp ( - s_0^\alpha e^{i\theta} |y|^{\frac{\alpha}{\alpha-1}}+  s_0 i |y|^{\frac{\alpha}{\alpha-1}}) + E_1(y) + E_2(y) ,
\end{equation}
where $0<\omega<\frac{\pi}{2}$ is a fixed number and $s_0=(\alpha \sin \theta)^{\frac{1}{1-\alpha}}$. Furthermore the following holds uniformly for $0< \omega \le |\theta|<\frac{\pi}{2}$,
$$  |C_{\theta, 1}| \le C_1;$$
$$ |E_1(y)|\le C_2, \hspace{1cm} \forall |y|\le1 ; $$
%$$ |E_2(y)|\le C_3 |y|^{-n\frac{1-\frac{\alpha}{2}}{1-\alpha} } \exp ( - C_4 \cos \theta |y|^{\frac{\alpha}{\alpha-1}}), \hspace{0.5cm} |y|\le1; $$
$$ |E_2(y)|\le \begin{cases}C_3 |y|^{-n\frac{1-\frac{\alpha}{2}}{1-\alpha}+\frac{\alpha}{2(1-\alpha)}} e^{-C_4\cos\theta |y|^{-\alpha}} & n\ge 2, \\   C_5 |\ln |y|| |y|^{-n\frac{1-\frac{\alpha}{2}}{1-\alpha}+\frac{\alpha}{2(1-\alpha)}} e^{-C_6\cos\theta |y|^{-\alpha}} & n= 1,    \end{cases} \hspace{1cm} \forall |y|\le 1,$$
where $C_i>0 $ for $1\le i \le 6$ are constants determined only by $\alpha, n, \omega$.

Moreover, for $N>0$ there exists constant $C_N>0$ such that
\begin{equation}\label{theorem1.2-2}
  |P_2(e^{i\theta}, y)| \le C_N |y|^{-N},  \hspace{1cm} \forall  |\theta|\le \frac{\pi}{2}, |y|\ge1.
\end{equation}

(3) When $\alpha>1$, for $0< \omega \le |\theta|<\frac{\pi}{2}$ there exists constant $C'_{\theta, 1}$ such that
\begin{equation}\label{theorem1.3}
  P_2(e^{i\theta}, y)= C'_{\theta, 1}|y|^{-n\frac{1-\frac{\alpha}{2}}{1-\alpha}} \exp ( - s_0^\alpha e^{i\theta} |y|^{\frac{\alpha}{\alpha-1}}+  s_0 i |y|^{\frac{\alpha}{\alpha-1}}) + E'_1(y) + E'_2(y),
\end{equation}
where $0<\omega<\frac{\pi}{2}$ is a fixed number. Furthermore, the following holds uniformly for $0< \omega \le |\theta|<\frac{\pi}{2}$,
$$  |C'_{\theta, 1}| \le C'_1 ;$$
$$ |E'_1(y)|\le C'_2 |y|^{-n-\alpha}, \hspace{1cm}\forall |y|\ge1 ; $$
$$ |E'_2(y)|\le C'_3 |y|^{-n\frac{1-\frac{\alpha}{2}}{1-\alpha} + \frac{\alpha}{2(1-\alpha)}} \exp ( - C'_4 \cos \theta |y|^{\frac{\alpha}{\alpha-1}}), \hspace{0.5cm} \forall|y|\ge1, $$
where $C'_i>0 $ for $1\le i \le 4$ are constants determined only by $\alpha, n, \omega$.

Moreover, there exists constant $C>0$ such that the following holds uniformly for $|\theta|\le \frac{\pi}{2}$
\begin{equation}\label{theorem1.3-2}
  |P_2(e^{i\theta}, y)| \le C,  \hspace{1cm} \forall |\theta|\le \frac{\pi}{2},  |y|\le 1.
\end{equation}
\end{proposition}

\begin{remark}
(1) Note that the term
$$ C|y|^{-n\frac{1-\frac{\alpha}{2}}{1-\alpha}} \exp ( - s_0^\alpha z |y|^{\frac{\alpha}{\alpha-1}}+  s_0 i |y|^{\frac{\alpha}{\alpha-1}}),$$
dominates the asymptotic behaviors of $P_2(e^{i\theta}, y)$ for $0<\alpha<1, |y|\rightarrow 0$ and $\alpha>1, |y|\rightarrow \infty$ respectively. Moreover, letting  $\theta \rightarrow \frac{\pi}{2}$ gives
 $$  |y|^{-n\frac{1-\frac{\alpha}{2}}{1-\alpha}} \exp ( - s_0^\alpha z |y|^{\frac{\alpha}{\alpha-1}}+  s_0 i |y|^{\frac{\alpha}{\alpha-1}} ) \rightarrow  |y|^{-n\frac{1-\frac{\alpha}{2}}{1-\alpha}} \exp( i \alpha^{\frac{\alpha}{1-\alpha}}(\alpha-1) |y|^{-\frac{\alpha}{\alpha-1}} ).  $$
Thus we have regained the asymptotic behaviors of $P(i, y)$(\cite[Proposition 5.1]{M}).

(2) For $ 0\le \theta \le \omega<\frac{\pi}{2} $, the following estimate holds:
\begin{equation}\label{remark1.2}
  |P(e^{i\theta}, y)|\le C 1 \wedge |y|^{-n-\alpha}, \hspace{0.8cm } \forall  \alpha>0, y\in \mathbb{R}^n.
\end{equation}
With little modification of the proof for  \eqref{theorem1.1} will show \eqref{remark1.2}.  Indeed, the estimates \eqref{theorem1.1} are well known and can be found in various literature(for example \cite{BG, HNS, K, MYZ, S}).
\end{remark}

\subsection{Proof  of Proposition 2.1 (1)}
\begin{proof}[Proof of \eqref{theorem1.1}]
Set
$$ L(y, D)= \frac{y\cdot \nabla_{\xi}}{i |y|^2} \hspace{1cm} \text{and} \hspace{1cm} L^*(y, D)=- \frac{y\cdot \nabla_{\xi}}{i |y|^2}. $$
It is direct to check $L(y, D) e^{iy\cdot \xi}= e^{i y \cdot \xi}$  and $L^*$ is the conjugate operator to $L$. Thus integration by parts gives
\begin{align*}
 & P_1(z, y)  \\
 = &  c_n \int_{\mathbb{R}^n} e^{iy\cdot \xi} L^*(e^{-z|\xi|^\alpha}  \varphi(|\xi|)) d\xi \\
 =  &  c_n \int_{\mathbb{R}^n} e^{iy\cdot \xi} \varphi\left(\frac{|\xi|}{\delta}\right)  L^*(e^{-z|\xi|^\alpha}  \varphi(|\xi|) ) d\xi\\
  &+  c_n \int_{\mathbb{R}^n} e^{iy\cdot \xi} \left(1- \varphi\left(\frac{|\xi|}{\delta}\right) \right) L^*(e^{-z|\xi|^\alpha}  \varphi(|\xi|) ) d\xi \\
  \triangleq &  I + II,
\end{align*}
where $\varphi$ is smooth cutoff and $\delta>0$ will be determined later.

For $I$, we obtain
\begin{align*}
  |I| & \le C \int_{|\xi|\le \delta} |L^* ( e^{-z|\xi|^\alpha}  \varphi(|\xi|))| d\xi  \\
      & \le \frac{C}{|y|} \int_{|\xi|\le \delta}  | \xi |^{\alpha-1}|\varphi(|\xi|)| +  |\varphi'(|\xi|)| d\xi\\
      & \le \frac{C}{|y|} \int_{|\xi|\le \delta}  | \xi |^{\alpha-1} d\xi = C' |y|^{-1} \delta^{\alpha+n-1}.
\end{align*}
In the last inequality, we have used the facts that $\varphi'(|\xi|)$ is supported in $\frac{1}{2}\le |\xi|\le 1$ and hence $|\varphi'(|\xi|)|\le C|\xi|^{\alpha-1}$ for some constant $C>0$.

Now we turn to the estimates of $II$. Integration by parts for  $N$ times with some $N>[\alpha]+n+1$ gives
\begin{align*}
  |II| & \le C \int_{\mathbb{R}^n} \left| (L^*)^{N-1} \left[\left(1-\varphi\left(\frac{|\xi|}{\delta}\right)\right)L^*(e^{-z|\xi|^\alpha} \varphi(|\xi|) ) \right] \right|   d\xi  \\
   & \le C \int_{|\xi|\ge \frac{\delta}{2}} | (L^*)^{N} (e^{-z|\xi|^\alpha} \varphi(|\xi|) )  |   d\xi \\
    &+ \sum_{k=1}^{N-1} C_k \int_{\mathbb{R}^n} \left|   (L^*)^k \left(1-\varphi\left(\frac{|\xi|}{\delta}\right)\right) (L^*)^{N-k}(e^{-z|\xi|^\alpha} \varphi(|\xi|) )  \right| d\xi.
\end{align*}
Since $\varphi(|\xi|)$ is supported in $|\xi|\le 1$ and $\varphi'(|\xi|)$ is supported in $\frac{1}{2}\le |\xi|\le 1$, then
$$ |(L^*)^N  (e^{-z|\xi|^\alpha} \varphi(|\xi|) )|  \le C_{N} |\xi|^{\alpha-N} |y|^{-N}. $$
Therefore
\begin{align*}
  |II| & \le \frac{C}{|y|^N} \left( \int_{|\xi|\ge  \frac{\delta}{2}} |\xi|^{\alpha-N} d\xi   + \sum_{k=1}^{N-1} \int_{| \frac{\delta}{2}\le |\xi|\le \delta}    \delta^{-k} |y|^{\alpha-N+k} d\xi \right) \\
       & \le C |y|^{-N} \delta^{\alpha-N+n}.
\end{align*}
Combing the estimates of $I$ and $II$ gives
$$| P_1(z, y)| \le C ( |y|^{-1}\delta^{\alpha+n-1} + |y|^{-N} \delta^{\alpha-N+n} ).  $$
Letting $\delta=|y|^{-1}$ implies
$$ | P_1(z, y)| \le C |y|^{-n-\alpha}.  $$
As a result, \eqref{theorem1.1} follows since $P_1(z, y)$ is bounded for $y\in \mathbb{R}^n$.
\end{proof}
\subsection{Proof of  of Proposition 2.1 (2)}
Note that it is sufficient to consider $0< \omega\le \theta <\frac{\pi}{2}$, since $\overline{P(z, y)}=P(\bar{z}, -y)$.
\begin{proof}[Proof of \eqref{theorem1.2}]
Since $\varphi(|\xi|)$ is supported in $|\xi|\le 1$, we have
\begin{equation}\label{3.2-1}
  P_2(z, y) = c_n \int_{\frac{1}{2}\le |\xi|\le 1} e^{i y\cdot \xi} e^{-z|\xi|^\alpha} (1-\varphi(|\xi|)) d\xi + c_n \int_{|\xi|\ge1} e^{i y\cdot \xi} e^{-z|\xi|^\alpha}  d\xi.
\end{equation}
It is clear that
$$  \left| \int_{\frac{1}{2}\le |\xi|\le 1} e^{i y\cdot \xi} e^{-z|\xi|^\alpha} (1-\varphi(|\xi|)) d\xi  \right| \le C , \hspace{1cm} \forall y\in \mathbb{R}^n, ~~~~ \omega\le \theta <\frac{\pi}{2}.    $$
Thus it is sufficient to consider
\begin{align*}
   & \int_{|\xi|\ge1} e^{i y\cdot \xi} e^{-z|\xi|^\alpha}  d\xi \\
  = &  C \int_1^{+\infty} e^{-zr^\alpha} r^{n-1} (r|y|)^{\frac{2-n}{2}} J_{\frac{n}{2}-1} (r|y|) dr\\
  = & C' |y|^{n\frac{1-\frac{\alpha}{2}}{\alpha-1} + \frac{\alpha}{\alpha-1}}  \int_{|y|^{\frac{1}{1-\alpha}}}^{+\infty} e^{-z|y|^{\frac{\alpha}{\alpha-1}} s^\alpha} s^{\frac{n}{2}}  J_{\frac{n}{2}-1} (s|y|^{\frac{\alpha}{\alpha-1}}) ds\\
  =& C' |y|^{n\frac{1-\frac{\alpha}{2}}{\alpha-1} } A \int_{A^{-\frac{1}{\alpha}}}^{+\infty} e^{-z A s^\alpha} s^{\frac{n}{2}}  J_{\frac{n}{2}-1} (sA) ds,
\end{align*}
where $A=|y|^{\frac{\alpha}{\alpha-1}}$ and we have changed the variable $r= |y|^{\frac{1}{\alpha-1}}s$ in the second equality. Note that $A\rightarrow +\infty$ as $|y|\rightarrow 0$ when $0<\alpha<1$.

Then we have
\begin{equation}\label{3.2-2}
  \int_{|\xi|\ge1} e^{i y\cdot \xi} e^{-z|\xi|^\alpha}  d\xi \triangleq C' |y|^{n\frac{1-\frac{\alpha}{2}}{\alpha-1} } A I,
\end{equation}
where
$$ I= \int_{A^{-\frac{1}{\alpha}}}^{+\infty} e^{-z A s^\alpha} s^{\frac{n}{2}}  J_{\frac{n}{2}-1} (sA) ds.  $$
To prove \eqref{theorem1.2}, we need further to estimate $I$.
\begin{equation}\label{3.2-3}
  I= \int_{A^{-\frac{1}{\alpha}}}^{A^{-1}} + \int_{A^{-1}}^{+\infty} \triangleq I_1 + I_2.
\end{equation}
First we have,
\begin{align*}
 I_1 = & \int_{A^{-\frac{1}{\alpha}}}^{A^{-1}} e^{-z A s^\alpha} s^{\frac{n}{2}}  J_{\frac{n}{2}-1} (sA) ds   \\
     = & A^{-\frac{n}{2}-1} \int_{A^{1-\frac{1}{\alpha}}}^{1} e^{-z A^{1-\alpha} t^\alpha} t^{\frac{n}{2}}  J_{\frac{n}{2}-1} (t) dt \\
     = & C A^{-\frac{n}{2}-1} \int_{A^{\alpha-1}}^{1} e^{-z A^{1-\alpha} \tau} \tau^{\frac{n}{2\alpha} + \frac{1}{\alpha} -1 }  J_{\frac{n}{2}-1} (\tau^{\frac{1}{\alpha}}) d\tau.
\end{align*}
According to \eqref{Bessel def}, we have
\begin{align*}
   & \int_{A^{\alpha-1}}^{1} e^{-z A^{1-\alpha} \tau} \tau^{\frac{n}{2\alpha} + \frac{1}{\alpha} -1 }  J_{\frac{n}{2}-1} (\tau^{\frac{1}{\alpha}}) d\tau \\
   =& \int_{A^{\alpha-1}}^{1} e^{-z A^{1-\alpha} \tau} \sum_{k\ge0} a_k \tau^{\frac{1}{\alpha}(n-\alpha+2k) }   d\tau\\
   =& a_0 \int_{A^{\alpha-1}}^{1} e^{-z A^{1-\alpha} \tau}  \tau^{\frac{1}{\alpha}(n-\alpha) }   d\tau + \sum_{k\ge 1} a_k \int_{A^{\alpha-1}}^{1} e^{-z A^{1-\alpha} \tau}  \tau^{\frac{1}{\alpha}(n-\alpha+2k) }   d\tau,
\end{align*}
where $a_k= \frac{(-1)^k 2^{1-2k-\frac{n}{2}}}{k! \Gamma(k+\frac{n}{2})}$.

Note that
$$ \int_{A^{\alpha-1}}^{1} e^{-z A^{1-\alpha} \tau}  \tau^{\frac{1}{\alpha}(n-\alpha) }   d\tau = \int_0^{+\infty}- \int_0^{A^{\alpha-1}} - \int_1^{+\infty}. $$
It is clear that
$$ \int_0^{+\infty} e^{-z A^{1-\alpha} \tau}  \tau^{\frac{1}{\alpha}(n-\alpha) }   d\tau =  \Gamma(\frac{n}{\alpha}) (zA^{1-\alpha})^{-\frac{n}{\alpha}}=  \Gamma(\frac{n}{\alpha}) z^{-\frac{n}{\alpha}} A^{\frac{n}{\alpha}(\alpha-1)},  $$
and
$$  \left|\int_0^{A^{\alpha-1}} e^{-z A^{1-\alpha} \tau}  \tau^{\frac{1}{\alpha}(n-\alpha) }   d\tau \right| \le C A^{\alpha-1} A^{(\alpha-1)\frac{1}{\alpha}(n-\alpha)}=C A^{\frac{n}{\alpha}(\alpha-1)}. $$
Moreover, integration by parts gives
\begin{align*}
   & \int_1^{+\infty} e^{-z A^{1-\alpha} \tau}  \tau^{\frac{1}{\alpha}(n-\alpha) }   d\tau  \\
  = & \sum_{k=1}^{[\frac{n}{\alpha}]+1} c_k \frac{e^{-zA^{1-\alpha}}}{(zA^{1-\alpha})^k}   + \frac{c_{[\frac{n}{\alpha}]+1}}{(zA^{1-\alpha})^{[\frac{n}{\alpha}]+1}}  \int_1^{+\infty} e^{-z A^{1-\alpha} \tau}  \tau^{\frac{n}{\alpha}-[\frac{n}{\alpha}]-2}   d\tau.
\end{align*}
Combing these estimates gives
\begin{equation}\label{3.2-4}
  \int_{A^{\alpha-1}}^{1} e^{-z A^{1-\alpha} \tau}  \tau^{\frac{1}{\alpha}(n-\alpha) }   d\tau = \Gamma(\frac{n}{\alpha}) z^{-\frac{n}{\alpha}} A^{\frac{n}{\alpha}(\alpha-1)}+ \sum_{k=1}^{[\frac{n}{\alpha}]+1} c_k \frac{e^{-zA^{1-\alpha}}}{(zA^{1-\alpha})^k}  + H_1(A),
\end{equation}
where $H_1(A)$ satisfies
$$ |H_1(A)| \le C A^{\frac{n}{\alpha}(\alpha-1)}, \hspace{1cm} \forall A\ge1, $$
for some positive constant $C>0$.

On the other hand, integration by parts gives
\begin{align*}
   & \int_{A^{\alpha-1}}^{1} e^{-z A^{1-\alpha} \tau}  \tau^{\frac{1}{\alpha}(n-\alpha+2k) }   d\tau   \\
  = & \frac{e^{-z}}{z} A^{(\alpha-1)\frac{1}{\alpha}(n+2k)} - \frac{e^{-zA^{1-\alpha}}}{zA^{1-\alpha}}+ \frac{n-\alpha+2k}{\alpha zA^{1-\alpha}} \int_{A^{\alpha-1}}^{1} e^{-z A^{1-\alpha} \tau}  \tau^{\frac{1}{\alpha}(n-\alpha+2k)-1 }   d\tau.
\end{align*}
After $[\frac{n}{\alpha}]+1$ steps of integrating by parts, we obtain
\begin{align*}
   & \int_{A^{\alpha-1}}^{1} e^{-z A^{1-\alpha} \tau}  \tau^{\frac{1}{\alpha}(n-\alpha+2k) }   d\tau \\
  = & A^{(\alpha-1)\frac{1}{\alpha}(n+2k)}e^{-z} \sum_{l=1}^{[\frac{n}{\alpha}]+1} c_l z^{-l} + e^{-zA^{1-\alpha}} \sum_{l=1}^{[\frac{n}{\alpha}]+1} c'_l (zA^{1-\alpha})^{-l}\\
  & + \frac{c_{[\frac{n}{\alpha}]+1}}{(zA^{1-\alpha})^{[\frac{n}{\alpha}]+1}}  \int_{A^{\alpha-1}}^1 e^{-z A^{1-\alpha} \tau}  \tau^{\frac{1}{\alpha}(n-\alpha+2k)-[\frac{n}{\alpha}]-1}   d\tau.
\end{align*}
It follows that
\begin{align*}
   & \sum_{k\ge1} a_k \int_{A^{\alpha-1}}^{1} e^{-z A^{1-\alpha} \tau}  \tau^{\frac{1}{\alpha}(n-\alpha+2k) }   d\tau  \\
  = & e^{-z} \sum_{l=1}^{[\frac{n}{\alpha}]+1} c_l z^{-l} \sum_{k\ge1} a_k A^{(\alpha-1)\frac{1}{\alpha}(n+2k)} + e^{-zA^{1-\alpha}} \sum_{l=1}^{[\frac{n}{\alpha}]+1} c'_l (zA^{1-\alpha})^{-l} \sum_{k\ge1} a_k\\
  &+ \frac{c_{[\frac{n}{\alpha}]+1}}{(zA^{1-\alpha})^{[\frac{n}{\alpha}]+1}} \sum_{k\ge1} a_k  \int_{A^{\alpha-1}}^1 e^{-z A^{1-\alpha} \tau}  \tau^{\frac{1}{\alpha}(n-\alpha+2k)-[\frac{n}{\alpha}]-1}   d\tau.
\end{align*}
Furthermore there exists $C>0$ determined by $\alpha, n$ such that
$$ \left|  \int_{A^{\alpha-1}}^1 e^{-z A^{1-\alpha} \tau}  \tau^{\frac{1}{\alpha}(n-\alpha+2k)-[\frac{n}{\alpha}]-1}   d\tau \right|\le C, \hspace{1cm} \forall k\ge1, ~~~~ A\ge 1.$$
Since $\sum\limits_{k\ge1} a_k <\infty$, we conclude
\begin{align}\label{3.2-5}
   & \sum_{k\ge1} a_k \int_{A^{\alpha-1}}^{1} e^{-z A^{1-\alpha} \tau}  \tau^{\frac{1}{\alpha}(n-\alpha+2k) }   d\tau \\
  = &  e^{-z} \sum_{l=1}^{[\frac{n}{\alpha}]+1} c_l z^{-l} A^{(\alpha-1)\frac{n}{\alpha}} \sum_{k\ge1} a_k A^{(\alpha-1)\frac{2k}{\alpha}} + e^{-zA^{1-\alpha}} \sum_{l=1}^{[\frac{n}{\alpha}]+1} C_l (zA^{1-\alpha})^{-l} + H_2(A), \notag
\end{align}
where $H_2(A)$ satisfies
$$ |H_2(A)|  \le C A^{\frac{n}{\alpha}(\alpha-1)}, \hspace{1cm} \forall A\ge1.  $$
Thus  \eqref{3.2-4} and \eqref{3.2-5} imply
\begin{align*}
  I_1 = & Cz^{-\frac{n}{\alpha}} A^{\frac{n}{2}-\frac{n}{\alpha}-1} + B_1(z) A^{\frac{n}{2}-\frac{n}{\alpha}-1} \sum_{k\ge1} a_k A^{(\alpha-1)\frac{2k}{\alpha}} \\
   & + e^{-zA^{1-\alpha}} A^{-\frac{n}{2}-1} \sum_{k=1}^{[\frac{n}{\alpha}]+1} c'_k (zA^{1-\alpha})^{-l} + A^{-\frac{n}{2}-1} (H_1(A) + H_2(A)),
\end{align*}
where
$$B_1(z)=e^{-z} \sum_{l=1}^{[\frac{n}{\alpha}]+1} c_l z^{-l},$$
as in \eqref{3.2-5}.

By the definition, $A=|y|^{\frac{\alpha}{\alpha-1}}$, we have $|y|^{n\frac{1-\frac{\alpha}{2}}{\alpha-1}}=A^{ \frac{n}{\alpha}-\frac{n}{2} }$ and hence
\begin{equation}\label{3.2-6}
  |y|^{n\frac{1-\frac{\alpha}{2}}{\alpha-1}} A I_1 = \tilde{E}_1(y) + \tilde{E}_2(y),
\end{equation}
where
\begin{equation}\label{3.2-7}
  |\tilde{E}_1(y)| \le C_1 , \hspace{0.3cm} |\tilde{E}_2(y)| \le C_2 |y|^{n\frac{1-\frac{\alpha}{2}}{\alpha-1}+ \frac{\alpha}{\alpha-1}(\alpha-1-\frac{n}{2}) } e^{-C_3\cos \theta |y|^{-\alpha}}, \hspace{0.5cm} \forall |y|\le 1
\end{equation}
for some constants $C_1, C_2, C_3>0$ determined only by $n, \alpha, \omega$.

Next we will employ the oscillatory integrals theory to deal with $I_2$.
By \eqref{Bessel pro2}, we have
\begin{align*}
  I_2= & \int_{A^{-1}}^{+\infty} e^{-zAs^\alpha} s^{\frac{n}{2}} J_{\frac{n}{2}-1} (sA) ds  \\
  = & A^{-\frac{1}{2}} \int_{A^{-1}}^{+\infty} e^{-z A s^\alpha +isA} s^{\frac{n-1}{2}}  L_1(sA) ds \\
   &+ A^{-\frac{1}{2}} \int_{A^{-1}}^{+\infty} e^{-z A s^\alpha -isA} s^{\frac{n-1}{2}}  L_2(sA) ds,
\end{align*}
where
$$ L_1(sA)= \sum_{k\ge 0} b_k (sA)^{-k},\hspace{1cm} L_2(sA)= \sum_{k\ge 0} b'_k (sA)^{-k},   $$
as in \eqref{Bessel pro2}.
To proceed, consider
\begin{equation}\label{3.2-8}
 \int_{A^{-1}}^{+\infty} e^{-z A s^\alpha +isA} s^{\frac{n-1}{2}}  L_1(sA) ds =  \int_{A^{-1}}^{\delta s_0} + \int_{\delta s_0}^{\frac{s_0}{\delta}} + \int_{\frac{s_0}{\delta}}^{+\infty} \triangleq J_1 + J_2 +J_3,
\end{equation}
where $s_0=(\alpha \sin \theta)^{\frac{1}{1-\alpha}}$, $\delta>0$ is close enough to 1 and will be determined later.

According to the definition of $L_1(sA)$, it follows
$$ J_1=b_0\int_{A^{-1}}^{\delta s_0} e^{-zAs^\alpha+ isA} s^{\frac{n-1}{2}} ds + A^{-1} \int_{A^{-1}}^{\delta s_0} e^{-zAs^\alpha+ isA} s^{\frac{n-3}{2}} \sum_{k\ge 1} b_k (sA)^{-k+1} ds.  $$
Note that for $A^{-1}\le s \le \delta s_0$, we have $1\le sA$ and hence
$$  \left| A^{-1} \int_{A^{-1}}^{\delta s_0} e^{-zAs^\alpha+ isA} s^{\frac{n-3}{2}} \sum_{k\ge 1} b_k (sA)^{-k+1} ds  \right| \le C A^{-1} e^{-\cos\theta A^{1-\alpha}} \int_{A^{-1}}^{\delta s_0} s^{\frac{n-3}{2}}  ds.  $$
Then we obtain, for $n\ge 2$
\begin{equation}\label{3.2-9}
  \left| A^{-1} \int_{A^{-1}}^{\delta s_0} e^{-zAs^\alpha+ isA} s^{\frac{n-3}{2}} \sum_{k\ge 1} b_k (sA)^{-k+1} ds  \right| \le C A^{-1} e^{-\cos\theta A^{1-\alpha}},
\end{equation}
 and for $n=1$,
\begin{equation}\label{3.2-10}
  \left| A^{-1} \int_{A^{-1}}^{\delta s_0} e^{-zAs^\alpha+ isA} s^{\frac{n-3}{2}} \sum_{k\ge 1} b_k (sA)^{-k+1} ds  \right| \le C' \frac{\ln A}{A} e^{-\cos\theta A^{1-\alpha}}.
\end{equation}
On the other hand,  integration by parts gives for $n\ge 2$
\begin{align*}
   & \int_{A^{-1}}^{\delta s_0} e^{-zAs^\alpha+ isA} s^{\frac{n-1}{2}} ds   \\
  = & A^{-1}\int_{A^{-1}}^{\delta s_0}  s^{\frac{n-1}{2}} h(s) de^{-zAs^\alpha+ isA}\\
  =&  A^{-1}[(\delta s_0)^{\frac{n-1}{2}} h(\delta s_0)  e^{-zA(\delta s_0)^\alpha+ i\delta s_0 A}- A^{-\frac{n-1}{2}} h(A^{-1}) e^{-zA^{1-\alpha}+ i} ] \\
  & - A^{-1} \int_{A^{-1}}^{\delta s_0} e^{-zAs^\alpha+ isA} \left[\frac{n-1}{2}s^{\frac{n-3}{2}}h(s)+ s^{\frac{n-1}{2}} h'(s)\right] ds,
\end{align*}
where $h(s)= (-e^{i\theta}\alpha s^{\alpha-1}+ i)^{-1} $. By Lemma 4.1, we have
\begin{align*}
   & \left| A^{-1} \int_{A^{-1}}^{\delta s_0} e^{-zAs^\alpha+ isA} \left[\frac{n-1}{2}s^{\frac{n-3}{2}}h(s)+ s^{\frac{n-1}{2}} h'(s)\right] ds  \right| \\
 \le  &  CA^{-1} e^{-\cos \theta A^{1-\alpha}} \int_{A^{-1}}^{\delta s_0} s^{\frac{n-3}{2}} |h(s)| + s^{\frac{n-1}{2}} |h'(s)| ds\\
 \le & C' A^{-1} e^{-\cos \theta A^{1-\alpha}}.
\end{align*}
Then we conclude for $n\ge 2$
\begin{equation}\label{3.2-11}
  \left|\int_{A^{-1}}^{\delta s_0} e^{-zAs^\alpha+ isA} s^{\frac{n-1}{2}} ds \right| \le C A^{-1} e^{-\cos\theta A^{1-\alpha}}.
\end{equation}
When $n=1$, similarly we have
\begin{align*}
   & \int_{A^{-1}}^{\delta s_0} e^{-zAs^\alpha+ isA}  ds   \\
  = & A^{-1}\int_{A^{-1}}^{\delta s_0}   h(s) de^{-zAs^\alpha+ isA}\\
  =&  A^{-1}[h(\delta s_0)  e^{-zA(\delta s_0)^\alpha+ i\delta s_0 A}-  h(A^{-1}) e^{-zA^{1-\alpha}+ i} ] \\
  & - A^{-1} \int_{A^{-1}}^{\delta s_0} e^{-zAs^\alpha+ isA} h'(s) ds.
\end{align*}
Since $|h'(s)|\le c s^{-\alpha}$, we conclude for $n=1$
\begin{equation}\label{3.2-12}
  \left|\int_{A^{-1}}^{\delta s_0} e^{-zAs^\alpha+ isA}  ds \right| \le C' A^{-1} e^{-\cos\theta A^{1-\alpha}}.
\end{equation}
By \eqref{3.2-9}, \eqref{3.2-10}, \eqref{3.2-11}, \eqref{3.2-12}, it follows
\begin{equation}\label{3.2-13}
  |J_1|\le \begin{cases}C A^{-1}e^{-\cos\theta A^{1-\alpha}} & n\ge 2; \\   C' \frac{\ln A}{A}e^{-\cos\theta A^{1-\alpha}} & n= 1,    \end{cases}
\end{equation}
for some constants $C, C'>0$ only determined by $n, \alpha, \omega$.

To estimates $J_3$, we separate the integral into two parts
\begin{align*}
  J_3= & \sum_{k=0}^{[\frac{n+1}{2}]+1} b_k A^{-k} \int_{\frac{s_0}{\delta}}^{\infty} e^{-zAs^\alpha +isA} s^{\frac{n-1}{2}-k} ds \\
   & + A^{-[\frac{n+1}{2}]-1}  \int_{\frac{s_0}{\delta}}^{\infty} e^{-zAs^\alpha +isA} s^{\frac{n-1}{2}-[\frac{n+1}{2}]-1} \sum_{k\ge [\frac{n+1}{2}]+1} b_k (sA)^{-k+[\frac{n+1}{2}]+1}   ds.
\end{align*}
It is clear that
\begin{align*}
   & \left| A^{-[\frac{n+1}{2}]-1}  \int_{\frac{s_0}{\delta}}^{\infty} e^{-zAs^\alpha +isA} s^{\frac{n-1}{2}-[\frac{n+1}{2}]-1} \sum_{k\ge [\frac{n+1}{2}]+1} b_k (sA)^{-k+[\frac{n+1}{2}]+1}  \right|  \\
  \le & C   A^{-[\frac{n+1}{2}]-1} e^{-\cos \theta A (\frac{s_0}{\delta})^\alpha } \int_{\frac{s_0}{\delta}}^{\infty} s^{\frac{n-1}{2}-[\frac{n+1}{2}]-1} ds\\
  \le & C'   A^{-[\frac{n+1}{2}]-1} e^{-\cos \theta A (\frac{s_0}{\delta})^\alpha }.
\end{align*}
Integration by parts for $N=[\frac{n+1}{2}]+1$ times gives
\begin{align*}
   & \int_{\frac{s_0}{\delta}}^{+\infty} e^{-zAs^\alpha +isA} s^{\frac{n-1}{2}-k} ds  \\
  = & e^{-zA(\frac{s_0}{\delta})^\alpha + iA\frac{s_0}{\delta}} \sum_{k=1}^{N} c_k A^{-k}\\
  - & A^{-N} \int_{\frac{s_0}{\delta}}^{+\infty} e^{-zAs^\alpha +isA} \sum_{\beta_1, \cdots, \beta_{N+1}} C_{\beta_1, \cdots , \beta_{N+1}} h^{(\beta_1)}(s) \cdots h^{(\beta_N)}(s) s^{\frac{n-1}{2}-k-\beta_{N+1}} ds,
\end{align*}
where $\beta_k \ge 0$ are integers satisfying $ \beta_1+ \cdots+ \beta_{N+1} =N $.
By Lemma 4.1, we obtain
\begin{align*}
   & \left| A^{-N} \int_{\frac{s_0}{\delta}}^{+\infty} e^{-zAs^\alpha +isA} \sum_{\beta_1, \cdots, \beta_{N+1}} C_{\beta_1, \cdots , \beta_{N+1}} h^{(\beta_1)}(s) \cdots h^{(\beta_N)}(s) s^{\frac{n-1}{2}-k-\beta_{N+1}} ds  \right| \\
  \le & C A^{-N} e^{-\cos \theta (\frac{s_0}{\delta})^\alpha} \int_{\frac{s_0}{\delta}}^{+\infty} s^{\frac{n-1}{2}-N-k} ds\\
  \le & C' A^{-N} e^{-\cos \theta (\frac{s_0}{\delta})^\alpha}.
\end{align*}
Therefore,
\begin{equation}\label{3.2-14}
  |J_3|\le C A^{-1} e^{-c\cos \theta A},
\end{equation}
where $C, c>0$ only determined by $n, \alpha, \omega$.

For $J_2$, we will apply the oscillatory integral theories. For this purpose, $J_2$ can be written as
\begin{align*}
  J_2= & b_0\int_{\delta s_0}^{\frac{s_0}{\delta}} e^{-iA(\sin \theta s^\alpha-s)} e^{-\cos \theta A s^\alpha} s^{\frac{n-1}{2}} ds \\
   & + A^{-1} \int_{\delta s_0}^{\frac{s_0}{\delta}}  e^{-zAs^\alpha +isA} s^{\frac{n-3}{2}}  \sum_{k\ge1} b_k (sA)^{-k+1} ds.
\end{align*}
It is clear that
$$   \left| A^{-1} \int_{\delta s_0}^{\frac{s_0}{\delta}}  e^{-zAs^\alpha +isA} s^{\frac{n-3}{2}}  \sum_{k\ge1} b_k (sA)^{-k} ds \right| \le C A^{-1} e^{-(\delta s_0)^\alpha \cos \theta A}.       $$
On the other hand,
\begin{align*}
   &  \int_{\delta s_0}^{\frac{s_0}{\delta}} e^{-iA(\sin \theta s^\alpha-s)} e^{-\cos \theta A s^\alpha} s^{\frac{n-1}{2}} ds    \\
  = & \int_{\delta s_0}^{\frac{s_0}{\delta}} e^{-iA(\sin \theta s^\alpha-s)} e^{-\cos \theta A s^\alpha} s^{\frac{n-1}{2}}(\eta_1(s) + \eta_2(s)) ds,
\end{align*}
where $\eta_1(s)$ is smooth, supported in $[\delta s_0,   \frac{s_0}{\delta}]$ and equals 1 for $s\in [\delta' s_0,   \frac{s_0}{\delta'}]$ with $\delta<\delta'$; $\eta_2(s)=1-\eta_1(s)$.

By stationary phase method(\cite[Proposition 3, p.334]{S}), letting $\delta'$ close enough to 1 implies
\begin{align*}
   &  \int_{\delta s_0}^{\frac{s_0}{\delta}} e^{-iA(\sin \theta s^\alpha-s)} e^{-\cos \theta A s^\alpha} s^{\frac{n-1}{2}} \eta_1(s) ds    \\
  = & e^{-i A(\sin \theta s_0^\alpha-s_0)} e^{-\cos \theta A (\delta s_0)^\alpha} \\
     & \times \int_{\delta s_0}^{\frac{s_0}{\delta}} e^{-iA[\sin \theta (s^\alpha-s_0^\alpha)-(s-s_0)]} e^{-\cos \theta A (s^\alpha-(\delta s_0)^\alpha)} s^{\frac{n-1}{2}} \eta_1(s) ds    \\
  = &   e^{-i A(\sin \theta s_0^\alpha-s_0)} e^{-\cos \theta A (\delta s_0)^\alpha} A^{-\frac{1}{2}} d_0 + H_3(A),
\end{align*}
where
$$ d_0=\left( \frac{2\pi}{-i\alpha(\alpha-1)\sin \theta s_0^{\alpha-2}}  \right)^{-\frac{1}{2}}  s_0^{\frac{n-1}{2}} e^{-\cos \theta A s_0^\alpha+\cos \theta A (\delta s_0)^\alpha},$$
and
$$ |H_3(A)|\le C A^{-1} e^{-c\cos \theta A}. $$
We have used the facts  for $k\ge 0$, there exists $C_k$ such that
$$ \left| \frac{d}{ds^k} e^{-\cos \theta A s^\alpha + \cos \theta A (\delta s_0)}   \right|\le C_k, \hspace{1cm} \forall A\ge 1, ~~~~0<\omega\le \theta<\frac{\pi}{2}. $$
Moreover, we have(\cite[Corollary. p.334]{S})
\begin{align*}
   & \left|\int_{\delta s_0}^{\frac{s_0}{\delta}} e^{-iA(\sin \theta s^\alpha-s)} e^{-\cos \theta A s^\alpha} s^{\frac{n-1}{2}} \eta_2(s) ds \right| \\
  \le  & C A^{-1} \left[ e^{-\cos \theta A (\frac{s_0}{\delta})^\alpha} \left(\frac{s_0}{\delta}\right)^{\frac{n-1}{2}} + \right. \\
   &\left.   \int_{\delta s_0}^{\frac{s_0}{\delta}}  e^{-\cos \theta A s^\alpha} s^{\frac{n-1}{2}} \left( \cos \theta A s^{ \alpha-1}\eta_2(s) +  s^{-1}\eta_2(s) + \eta'_2(s) \right) ds \right] \\
  \le & C A^{-1} \left( e^{-\cos \theta A (\frac{s_0}{\delta})^\alpha} + e^{-\cos \theta A (\delta s_0)^\alpha} \cos \theta A   \right) \\
  \le & C A^{-1} e^{-\frac{1}{2}\cos \theta A (\delta s_0)^\alpha}.
\end{align*}
As a result, we have
\begin{equation}\label{3.2-15}
  J_2=C_{\theta, 1} A^{-\frac{1}{2}} e^{-zA s_0^\alpha+i A s_0} + H_4(A), \hspace{0.3cm} \text{with} \hspace{0.3cm} |H_4(A)|\le C A^{-1} e^{-c\cos \theta A},
\end{equation}
where $C, c>0$ are determined by $n, \alpha, \omega$.

Since there is no critical point, i.e. $ |i \sin \theta A s^\alpha + i s A |\ge \sin \omega A^{1-\alpha} +1 >0 $ for $s\ge A^{-1}$, we can use integration by parts to estimates
$$  \int_{A^{-1}}^{+\infty} e^{-zAs^\alpha-isA} s^{\frac{n-1}{2}} L_2{sA} ds = \int_{A^{-1}} ^1 + \int_1^{+\infty} \triangleq J'_1 + J'_2.   $$
Following the arguments for $J_1, J_3$,  similarly we obtain
\begin{equation}\label{3.2-16}
  |J'_1|\le \begin{cases}C_1 A^{-1}e^{-\cos\theta A^{1-\alpha}} & n\ge 2; \\   C_2 \frac{\ln A}{A}e^{-\cos\theta A^{1-\alpha}} & n= 1,    \end{cases}
\end{equation}
as well as
\begin{equation}\label{3.2-17}
  |J'_2|\le C_3 A^{-1} e^{-c\cos \theta A},
\end{equation}
where $C_1, C_2, C_3, c>0$ determined only by $n, \alpha$.

As a result, by \eqref{3.2-8},\eqref{3.2-13},\eqref{3.2-14},\eqref{3.2-15},\eqref{3.2-16}, \eqref{3.2-17}, we conclude that
\begin{equation}\label{3.2-18}
  |y|^{n\frac{1-\frac{\alpha}{2}}{\alpha-1}} A I_2= C_{\theta ,1} |y|^{n\frac{1-\frac{\alpha}{2}}{\alpha-1}} e^{-z |y|^{\frac{\alpha}{\alpha-1}} s_0^\alpha +i  |y|^{\frac{\alpha}{\alpha-1}} s_0} + \tilde{E}_3(y), + \tilde{E}_4(y),
\end{equation}
where
\begin{equation}\label{3.2-19}
  |\tilde{E}_3(y)|\le C_1 |y|^{n\frac{1-\frac{\alpha}{2}}{\alpha-1}+\frac{\alpha}{2(1-\alpha)}} e^{-c_1 \cos \theta |y|^{\frac{\alpha}{\alpha-1}} }, \hspace{1cm} |y|\le 1,
\end{equation}
and
\begin{equation}\label{3.2-20}
  |\tilde{E}_4(y)|\le \begin{cases}C_3 |y|^{n\frac{1-\frac{\alpha}{2}}{\alpha-1}+\frac{\alpha}{2(1-\alpha)}} e^{-\cos\theta |y|^{-\alpha}} & n\ge 2; \\   C_4 |\ln |y|| |y|^{n\frac{1-\frac{\alpha}{2}}{\alpha-1}+\frac{\alpha}{2(1-\alpha)}} e^{-\cos\theta |y|^{-\alpha}} & n= 1,    \end{cases} \hspace{1cm} |y|\le 1.
\end{equation}
Finally we have shown \eqref{theorem1.2} through \eqref{3.2-6},\eqref{3.2-7},\eqref{3.2-18},\eqref{3.2-19},\eqref{3.2-20}.

\end{proof}
\begin{proof}[Proof of \eqref{theorem1.2-2}]
Indeed, \eqref{theorem1.2-2} follows easily from the arguments in \cite[p.52]{W}. To be more precious, since
$$ [(1-\varphi(|\xi|))e^{-z|\xi|^\alpha}]^\vee(x)=c |x|^{-2}[\Delta (1-\varphi(|\xi|))e^{-z|\xi|^\alpha}]^\vee(x),  $$
and for $0<\alpha<1, k> \frac{n}{2(1-\alpha)}$
$$ \int_{\mathbb{R}^n} \left| \Delta^k  (1-\varphi(|\xi|))e^{-z|\xi|^\alpha} \right| d\xi <+\infty,  $$
we have proved \eqref{theorem1.2-2}.
\end{proof}
\subsection{Proof  of  Proposition 2.1 (3)}
\begin{proof}[Proof of \eqref{theorem1.3}]
For simplicity, set $\psi(|\xi|)=1-\varphi(|\xi|)$ and $P_2(z, y)$ can be written as
$$ P_2(z, y)=c_n \int_{|\xi|\ge \frac{1}{2}} e^{iy\cdot \xi} e^{-z|\xi|^\alpha} \psi(|\xi|) d\xi. $$
Note that we can not separate the integral into $\int_{\frac{1}{2}\le |\xi|\le 1} + \int_{|\xi|\ge 1}$ as in \eqref{3.2-1} to simplify our proof. This is because the integrand at $|\xi|=1$ does not decay as $|y|\rightarrow +\infty$ and hence the endpoint is hard to deal with after integrating by parts. For our purposes,
\begin{align*}
   & \int_{|\xi|\ge \frac{1}{2}} e^{iy\cdot \xi} e^{-z|\xi|^\alpha} \psi(|\xi|) d\xi \\
  = & C |y|^{n\frac{1-\frac{\alpha}{2}}{\alpha-1}} A \int_{\frac{1}{2}A^{-\frac{1}{\alpha}}}^{+\infty} \psi(sA^{\frac{1}{\alpha}}) e^{-zAs^\alpha} s^{\frac{n}{2}} J_{\frac{n}{2}-1} (sA) ds \\
  =& C |y|^{n\frac{1-\frac{\alpha}{2}}{\alpha-1}} A^{\frac{1}{2}} \int_{\frac{1}{2}A^{-\frac{1}{\alpha}}}^{+\infty} \psi(sA^{\frac{1}{\alpha}}) e^{-zAs^\alpha} e^{isA} s^{\frac{n-1}{2}} L_{1} (sA) ds \\
  & + C |y|^{n\frac{1-\frac{\alpha}{2}}{\alpha-1}} A^{\frac{1}{2}} \int_{\frac{1}{2}A^{-\frac{1}{\alpha}}}^{+\infty} \psi(sA^{\frac{1}{\alpha}}) e^{-zAs^\alpha} e^{-isA} s^{\frac{n-1}{2}} L_{2} (sA) ds,
\end{align*}
where $A=|y|^{\frac{\alpha}{\alpha-1}}$ and
$$ L_1(sA)=\sum_{k\ge 0} b_k (sA)^{-k}, \hspace{1cm} L_2(sA)=\sum_{k\ge 0} b'_k (sA)^{-k}. $$
Observe that $A\rightarrow +\infty$ as $|y|\rightarrow +\infty$ for $\alpha>1$.

To start with, consider
$$ \int_{\frac{1}{2}A^{-\frac{1}{\alpha}}}^{+\infty} \psi(sA^{\frac{1}{\alpha}}) e^{-zAs^\alpha+isA}  s^{\frac{n-1}{2}} L_{1} (sA) ds = \int_{\frac{1}{2}A^{-\frac{1}{\alpha}}}^{\delta s_0} + \int_{\delta s_0}^{\frac{s_0}{\delta}} + \int_{\frac{s_0}{\delta}}^{+\infty} \triangleq I_1 + I_2 + I_3,  $$
where $s_0=(\alpha \sin \theta)^{\frac{1}{1-\alpha}}$ and $\delta$ will be determined later.
Set $N_1=[\alpha+ \frac{n+1}{2}]+1$ and we have
\begin{align*}
  I_1 = & \sum_{k=0}^{N_1} b_k A^{-k} \int_{\frac{1}{2}A^{-\frac{1}{\alpha}}}^{\delta s_0} \psi(sA^{\frac{1}{\alpha}}) e^{-zAs^\alpha+isA}  s^{\frac{n-1}{2}-k} ds \\
   & + A^{-N_1} \int_{\frac{1}{2}A^{-\frac{1}{\alpha}}}^{\delta s_0}  \psi(sA^{\frac{1}{\alpha}}) e^{-zAs^\alpha+isA} s^{\frac{n-1}{2}-N_1} \sum_{k\ge N_1} b_k (sA)^{-k+N_1} ds.
\end{align*}
For $0\le k \le N_1$, integrating by parts $N_1$ times gives
\begin{align*}
   &A^{-k} \int_{\frac{1}{2}A^{-\frac{1}{\alpha}}}^{\delta s_0} \psi(sA^{\frac{1}{\alpha}}) e^{-zAs^\alpha+isA}  s^{\frac{n-1}{2}-k} ds   \\
  = & A^{-k} e^{-zA(\delta s_0)^\alpha+ iA \delta s_0} \sum_{l=1}^{N_1} C_l A^{-l} \\
  & + C'_{N_1} A^{-k-N_1} \int_{\frac{1}{2}A^{-\frac{1}{\alpha}}}^{\delta s_0} \sum_{\beta_1, \cdots, \beta_{N_1+2}} C_{\beta_1, \cdots, \beta_{N_1+2}} A^{\frac{\beta_1}{\alpha}} \psi^{(\beta_1)} (sA^{\frac{1}{\alpha}}) s^{\frac{n-1}{2}-k-\beta_2} \times \\
  & h^{(\beta_3)}(s) \cdots  h^{(\beta_{N_1+2})}(s) e^{-zAs^\alpha+isA} ds \\
  \triangleq & A^{-k} e^{-zA(\delta s_0)^\alpha+ iA \delta s_0} \sum_{l=1}^{N_1} C_l A^{-l} + H_4(A),
\end{align*}
where $\beta_k\ge 0$ are integers satisfying $\beta_1+ \cdots + \beta_{N_1+2} = N_1$ and $h(s)=(-\alpha z s^{\alpha-1} + i )^{-1}$. By Lemma 4.1, we obtain
$$ |H_4(A)| \le C A^{-k-N_1+\frac{\beta_1}{\alpha}}  \int_{\frac{1}{2}A^{-\frac{1}{\alpha}}}^{\delta s_0} \sum_{\beta_1, \cdots, \beta_{N_1+2}} |\psi^{(\beta_1)} (sA^{\frac{1}{\alpha}})| s^{\frac{n-1}{2}-k-\beta_2-\cdots - \beta_{N_1+2}} ds.  $$
When $\beta_1=0$, it implies
\begin{align*}
  |H_4(A)| &\le C A^{-k-N_1}  \int_{\frac{1}{2}A^{-\frac{1}{\alpha}}}^{\delta s_0}  s^{\frac{n-1}{2}-k-N_1} ds  \\
           & \le C A^{-k-N_1} A^{-\frac{1}{\alpha}(\frac{n-1}{2}-k-N_1+ 1)} \int_{\frac{1}{2}}^{\delta s_0 A^{\frac{1}{\alpha}}}  s^{\frac{n-1}{2}-k-N_1+ \beta_1} ds\\
           &\le C A^{-\frac{n+1}{2} + (1-\frac{1}{\alpha})(\frac{n+1}{2}-k-N_1) }\\
           & \le C A^{-\frac{n-1}{2}-\alpha}.
\end{align*}
We have used the facts $N_1-\frac{n+1}{2}\ge \alpha$ in the last inequality.

When $\beta_1 \ge 1$, we have
\begin{align*}
  |H_4(A)| &\le C A^{-k-N_1+\frac{\beta_1}{\alpha}}  \int_{\frac{1}{2}A^{-\frac{1}{\alpha}}}^{A^{-\frac{1}{\alpha}}}   s^{\frac{n-1}{2}-k-N_1+\beta_1} ds  \\
           & \le C A^{-k-N_1+\frac{\beta_1}{\alpha}} A^{-\frac{1}{\alpha}(\frac{n-1}{2}-k-N_1+ \beta_1 +1)}\\
            &\le C A^{-\frac{n+1}{2} + (1-\frac{1}{\alpha})(\frac{n+1}{2}-k-N_1) }\\
           & \le C A^{-\frac{n-1}{2}-\alpha}.
\end{align*}
As a result, the following estimate holds for $|H_4|(A)$,
\begin{equation}\label{3.2-21}
  |H_4(A)|\le C A^{-\frac{n-1}{2}+\alpha},  \hspace{1cm} \forall A\ge1.
\end{equation}
Together with the following estimates
\begin{align*}
   & A^{-N_1} \left| \int_{\frac{1}{2}A^{-\frac{1}{\alpha}}}^{\delta s_0}  \psi(sA^{\frac{1}{\alpha}}) e^{-zAs^\alpha+isA} s^{\frac{n-1}{2}-N_1} \sum_{k\ge N_1} b_k (sA)^{-k+N_1} ds \right| \\
  \le & C  A^{-N_1} \int_{\frac{1}{2}A^{-\frac{1}{\alpha}}}^{\delta s_0} s^{\frac{n-1}{2}-N_1}  ds\\
  \le & C A^{(\frac{1}{\alpha}-1)N_1-\frac{n+1}{2\alpha}} \le C A^{-\frac{n-1}{2}+\alpha},
\end{align*}
\eqref{3.2-21} implies
$$  I_1=  e^{-zA(\delta s_0)^\alpha+ iA \delta s_0} \sum_{k=1}^{N_1} C_k A^{-k} + H_5(A),  $$
and
$$ |H_5(A)| \le C  A^{-\frac{n-1}{2}+\alpha},  \hspace{1cm} \forall A\ge1. $$
In turn, we obtain that
\begin{equation}\label{3.2-22}
  |y|^{n\frac{1-\frac{\alpha}{2}}{\alpha-1}} A^{\frac{1}{2}} I_1 =  |y|^{n\frac{1-\frac{\alpha}{2}}{\alpha-1}} A^{-\frac{1}{2}} e^{-zA(\delta s_0)^\alpha+ iA \delta s_0} + \bar{E}_1 (A)+ \bar{E}_2(A),
\end{equation}
where
\begin{equation}\label{3.2-23}
  |\bar{E}_1(A)| \le C_1 |y|^{-n\frac{1-\frac{\alpha}{2}}{1-\alpha}+\frac{3\alpha}{2(1-\alpha)}} e^{-C_2 \cos \theta |y|^{\frac{\alpha}{\alpha-1}}}, \hspace{0.5cm} \forall |y|\ge1,
\end{equation}
and
\begin{equation}\label{3.2-24}
  |\bar{E}_2(A)| \le C_3 |y|^{-n-\alpha} \hspace{1cm} \forall |y|\ge1,
\end{equation}
where $C_1, C_2, C_3>0$ are only determined by $n, \alpha, \omega$.

Since $\psi(sA^{\frac{1}{\alpha}})=1$  for $  s\ge A^{-\frac{1}{\alpha}} $, then for $|A|\gg 1$ we have
$$ I_2=\int_{\delta s_0}^{\frac{s_0}{\delta}} e^{-zAs^{\alpha} + i sA} s^{\frac{n-1}{2}} L_1(sA) ds~~~~ \text{and} ~~~~ I_3=\int_{\frac{s_0}{\delta}}^{+\infty} e^{-zAs^{\alpha} + i sA} s^{\frac{n-1}{2}} L_1(sA) ds.   $$
Then the proof are almost the same as in the case $0<\alpha<1$ and we omit the details. It follows that
\begin{equation}\label{3.2-25}
  |I_3|\le C_1 A^{-1} e^{-C_2  \cos \theta A};
\end{equation}
\begin{equation}\label{3.2-26}
  I_2= C_{\theta, 1} A^{-\frac{1}{2}}  e^{ -z A s_0^\alpha + iA s_0 } + H_5(A)  \hspace{0.5cm} \text{with} \hspace{0.5cm} |H_5(A)|\le C_3 A^{-1} e^{-C_4 \cos \theta A},
\end{equation}
for some constants $C_1, C_2, C_3, C_4>0$ only determined by $n, \alpha, \omega$.

The estimates for
$$ C |y|^{n\frac{1-\frac{\alpha}{2}}{\alpha-1}} A^{\frac{1}{2}} \int_{\frac{1}{2}A^{-\frac{1}{\alpha}}}^{+\infty} \psi(sA^{\frac{1}{\alpha}}) e^{-zAs^\alpha} e^{-isA} s^{\frac{n-1}{2}} L_{2} (sA) ds, $$
are easer than the above proof due to the facts there is no critical points. The proof are minor correction to the above arguments and we omit the detail. Combing \eqref{3.2-22}, \eqref{3.2-23}, \eqref{3.2-24}, \eqref{3.2-25}, \eqref{3.2-26} implies \eqref{theorem1.3}.
\end{proof}
\begin{proof}[Proof of \eqref{theorem1.3-2}]
In fact, \eqref{theorem1.3-2} can be shown by Laplace transform. Firstly,
\begin{align*}
  P(z, y) & = c_n  |y|^{1-\frac{n}{2}} \int_0^{+\infty} e^{-zr^{\alpha}} r^{\frac{n}{2}} J_{\frac{n}{2}-1} (r|y|) dr  \\
   & = C |y|^{1-\frac{n}{2}} \int_0^{+\infty} e^{-zs} s^{\frac{n}{2\alpha}+\frac{1}{\alpha}-1} J_{\frac{n}{2}-1} (s^{\frac{1}{\alpha}}|y|) ds.
\end{align*}
By \eqref{Bessel def}, we have
\begin{align*}
  P(z, y) & =C |y|^{1-\frac{n}{2}} \int_0^{+\infty} e^{-zs} s^{\frac{n}{2\alpha}+\frac{1}{\alpha}-1} \sum_{k\ge0} \frac{(-1)^k}{k! \Gamma(k+\frac{n}{2})} \left( \frac{s^{\frac{1}{\alpha}} |y| }{2}  \right)^{2k+\frac{n}{2}-1} ds \\
          & =C \sum_{k\ge 0} \frac{(-1)^k 2^{-2k-\frac{n}{2}+1}}{k! \Gamma(k+\frac{n}{2})} |y|^{2k} \int_0^{+\infty} e^{-zs} s^{\frac{n+2k}{\alpha}-1}ds \\
          & = C z^{-\frac{n}{\alpha}} \sum_{k\ge 0} \frac{(-1)^k \Gamma(\frac{n+2k}{\alpha})}{4^k k! \Gamma(k+\frac{n}{2})} z^{-\frac{2k}{\alpha}} |y|^{2k}.
\end{align*}
The converge radius of above series is $(0, +\infty)$  for $\alpha>1$. Together with \eqref{theorem1.1}, the above implies \eqref{theorem1.3-2}.
\end{proof}
\section{Proof of Theorem 1.1 and Theorem 1.3}
\begin{proof}[Proof of Theorem 1.1]
Set $\omega=\frac{\pi}{4}$. In view of Proposition 2.1, for $0<\alpha<1$, $|y|\ge 1$, $\frac{\pi}{4}\le |\theta|<\frac{\pi}{2}$, we have
$$ |P(e^{i\theta}, y)|   \le C |y|^{-n-\alpha}. $$
On the other hand, by \eqref{theorem1.2} and \eqref{theorem1.2-2}, we obtain for $0<\alpha<1$, $|y|\le 1$, $\frac{\pi}{4}\le |\theta|<\frac{\pi}{2}$
$$ |P(e^{i\theta}, y)|   \le  C(1 + |y|^{-n\frac{1-\frac{\alpha}{2}}{1-\alpha}} e^{-c\cos \theta |y|^{\frac{\alpha}{\alpha-1}}}  )   $$
%\begin{align*}
 %|P(e^{i\theta}, y)|  & \le  C(1 + |y|^{n\frac{1-\frac{\alpha}{2}}{\alpha-1}} e^{-c\cos \theta |y|^{\frac{\alpha}{\alpha-1}}}  )  \\
%            & \le C (1+ (\cos \theta)^{-\frac{n}{\alpha}+ \frac{n}{2}} )\\
%            & \le C (\cos \theta)^{-\frac{n}{\alpha}+ \frac{n}{2}}.
%\end{align*}
%In the second step, we have used the facts
%$$  t^{\gamma}e^{-p t}\le  p^{-\gamma}, \hspace{1cm} \forall t, \gamma>0.  $$
Therefore, by \eqref{remark1.2} we obtain for $0<\alpha<1$, $0\le |\theta|<\frac{\pi}{2}$,
$$|P(e^{i\theta}, y)| \le \begin{cases} C_1  (1 + |y|^{-n\frac{1-\frac{\alpha}{2}}{1-\alpha}} e^{-c\cos \theta |y|^{\frac{\alpha}{\alpha-1}}}  ) , & |y|\le 1; \\ C_2  |y|^{-n-\alpha}, & |y|> 1.  \end{cases}$$
Since $P(z, x)=|z|^{-\frac{n}{\alpha}} P(e^{i\theta}, y)  $ with $y=\frac{x}{|z|^{\frac{1}{\alpha}}}$, \eqref{theorem2.1} follows.

When $\alpha>1$, by Proposition 2.1, we have for $\alpha>1$, $|y|\ge 1$, $ \frac{\pi}{4}\le |\theta|<\frac{\pi}{2} $
$$ |P(e^{i\theta}, y)|   \le  C(|y|^{-n-\alpha} + |y|^{-n\frac{1-\frac{\alpha}{2}}{1-\alpha}} e^{-c\cos \theta |y|^{\frac{\alpha}{\alpha-1}}}  ).  $$
%\begin{align*}
% |P(e^{i\theta}, y)|  & \le  C(|y|^{-n-\alpha} + |y|^{n\frac{1-\frac{\alpha}{2}}{\alpha-1}} e^{-c\cos \theta |y|^{\frac{\alpha}{\alpha-1}}}  )  \\
%            & \le C |y|^{-n-\alpha}(1+ |y|^{n+\alpha+n\frac{1-\frac{\alpha}{2}}{\alpha-1}} e^{-c\cos \theta |y|^{\frac{\alpha}{\alpha-1}}} )\\
 %           & \le C |y|^{-n-\alpha} (\cos \theta)^{-\frac{n}{2}-\alpha+1}.
%\end{align*}
In turn, combining the estimates \eqref{remark1.2} we conclude for $\alpha>1$, $0\le |\theta|<\frac{\pi}{2}$
$$|P(e^{i\theta}, y)| \le \begin{cases} C_1,   & |y|\le 1; \\ C_2  (|y|^{-n-\alpha} + |y|^{-n\frac{1-\frac{\alpha}{2}}{1-\alpha}} e^{-c\cos \theta |y|^{\frac{\alpha}{\alpha-1}}}  ),  & |y|> 1.  \end{cases}$$
And hence \eqref{theorem2.2} follows.
\end{proof}

%%%%%%%%%%%%%%%%%%%%%%%%%%%%%%%%%%%%%%%%%%%%%%%%%%%%%%%%%%%%%%%%%%%%%%%%%%%%%%%%%%%%%%%%%%%%

Now we are ready to  consider the fractional Schr\"odinger operator with Kato potentials. We adopt  the methods in \cite{BJ, HWZD} to prove Theorem 1.3. Set
$$ I(|z|,x)= |z|^{-\frac{n}{\alpha}} \wedge \frac{|z|}{|x|^{n+\alpha}}. $$
According to \eqref{theorem2.1}, we have
\begin{align*}
  |P(z, x)| & \le C \left[|z|^{-\frac{n}{\alpha}}+ \frac{|z|^{ \frac{n}{2(1-\alpha)} }  }{ |x|^{n\frac{1-\frac{\alpha}{2}}{1-\alpha}} } \exp \left( -C_2 \frac{|x|^{\frac{\alpha}{\alpha-1}}}{|z|^{\frac{1}{\alpha-1}}} \cos \theta \right)  \right]  \wedge \frac{|z|}{|x|^{n+\alpha}}  \\
   & \le C \left[  |z|^{-\frac{n}{\alpha}}+ |z|^{-\frac{n}{\alpha}}(\cos \theta)^{-\frac{n}{\alpha} + \frac{n}{2}} \right] \wedge \frac{|z|}{|x|^{n+\alpha}}\\
   & \le C (\cos \theta)^{-\frac{n}{\alpha} + \frac{n}{2}} |z|^{-\frac{n}{\alpha}}\wedge \frac{|z|}{|x|^{n+\alpha}}.
\end{align*}
In the second step, we have used the facts
$$  t^{\gamma}e^{-p t}\le  p^{-\gamma}, \hspace{1cm} \forall t, \gamma, p>0.  $$
Then there exist constants $D_1, D_2>0$ depending only on $n, \alpha$ such that
\begin{equation}\label{cnostant1}
  |P(z, x)| \le D_1 (\cos \theta)^{-\frac{n}{\alpha}+\frac{n}{2}} I(|z|, x), \hspace{1cm}  0<\alpha<1, ~~~~\forall z\in \mathbb{C}^+, x\in \mathbb{R}^n,
\end{equation}
and
\begin{equation}\label{cnostant2}
  |P(z, x)| \le D_2 (\cos \theta)^{-\frac{n}{2}-\alpha+1}  I( |z|, x), \hspace{1cm}  \alpha>1, ~~~~\forall z\in \mathbb{C}^+, x\in \mathbb{R}^n.
\end{equation}
Next we only prove \eqref{theorem3.1} in details cause minor correction of the proof will show \eqref{theorem3.2}.

Following \cite{HWZD}, we need some characterizations of Kato potentials.
\begin{lemma}
$V\in K_{\alpha}(\mathbb{R}^n)$ if and only if $\lim_{t\rightarrow 0} K_V(t)=0$, where
$$ K_V(t)= \sup_x  \int_{\mathbb{R}^n} J(t, x-y) |V(y)|dy, $$
and
$$ J(t,x)=\begin{cases}|x|^{\alpha-n} \wedge t^2|x|^{-n-\alpha}, &  0<\alpha<n, \\ (1\vee \ln (t|x|^{-n})) \wedge t^2|x|^{-2n}, & \alpha=n, \\ t^{1-n/\alpha} \wedge t^2|x|^{-n-\alpha}, & \alpha>n.  \end{cases} $$
\end{lemma}
\begin{proof}
The proof can be found in \cite{HWZD}.
\end{proof}
Denote by $\tilde{H}=e^{i\theta}(-\Delta)^{\frac{\alpha}{2}}+e^{i\theta}V$. Then we have $e^{-z((-\Delta)^{\frac{\alpha}{2}}+V)}=e^{-|z|\tilde{H}}$. To start with, set
$$ \tilde{K}_j(|z|,x, y)=\int_{\mathbb{R}^n} \int_0^{|z|} \tilde{K}_{j-1}(|z|-s, x, \zeta) e^{i\theta} V(\zeta) \tilde{K}_0(s, \zeta, y) ds d\zeta, \hspace{1cm} j\in \mathbb{N}^*,  $$
where $\tilde{K}_0(|z|, x, y)= P(z, x-y)$.

Then we have the following estimate for $\tilde{K}_j(|z|, x, y).$
\begin{lemma}
Let $0< \alpha< 1$. There exists a constant $\omega$ depending on $n, \alpha$ such that the following holds for $j\in \mathbb{N}^*$
$$ |\tilde{K}_j(|z|,x,y)| \le D_1 (w\tilde{K}_V(|z|))^j \tilde{I}(|z|, x-y), $$
where $\tilde{I}(|z|, x-y)=\eta I(|z|, x),~~ \tilde{K}_V(|z|)= \eta  K_V( |z|)$ for $\eta= (\cos \theta)^{-\frac{n}{\alpha}+ \frac{n}{2}}$ and $D_1$ is the constant in \eqref{cnostant1}.
\end{lemma}
\begin{proof}
When $j=0$, it is just \eqref{cnostant1}.\\
Note first that
\begin{align*}
 \tilde{I}(|z|,x) \wedge \tilde{I} (s, y)   & = \eta ( I( |z|, x)\wedge I( s, y)  )  \\
   & \le D_3\eta I((|z|+s), x+y)=D_3 \tilde{I}(|z|+s, x+y),
\end{align*}
where $D_3=2^{\alpha-1} \vee 2^{\frac{n}{2\alpha}}$. And hence
\begin{align*}
 \tilde{I}(|z|,x) \tilde{I}(s,y)  & =( \tilde{I}(|z|,x) \wedge \tilde{I}(s,y) )(\tilde{I}(|z|,x) \vee \tilde{I}(s,y) )  \\
   & \le D_3 \tilde{I}(|z|+s, x+y) ( \tilde{I}(|z|,x) \vee \tilde{I}(s,y)).
\end{align*}
Moreover, we have
\begin{align*}
 \int_0^{|z|} \tilde{I} (|z|-s,x) ds =\int_0^{|z|} \tilde{I} (s,x) ds & \le e\int_0^\infty e^{-\frac{s}{|z|}} \tilde{I} (s,x) ds  \\
   &  =e\eta \int_0^\infty e^{-\frac{s}{|z|}} I( s, x)ds\\
   & \le e D_4 \eta J(|z|, x).
\end{align*}
The proof of the last inequality can be found in \cite{HWZD}. Then we have by induction,
\begin{align*}
  & |\tilde{K}_j(|z|,x,y)|\\
  \le &  D_1^2 (w\tilde{K}_V)^{j-1} \int_{\mathbb{R}^n} \int_0^{|z|} \tilde{I}(|z|-s, x-\zeta) \wedge \tilde{I}(s, \zeta-y) |V(\zeta)| dsd\zeta  \\
 \le &  D_1^2 D_3 (w\tilde{K}_V)^{j-1} \tilde{I} (|z|, x-y) \int_{\mathbb{R}^n} \int_0^{|z|} \tilde{I}(|z|-s, x-\zeta) \vee \tilde{I}(s, \zeta-y) |V(\zeta)| dsd\zeta \\
  \le &  eD_1^2 D_3 D_4 (w\tilde{K}_V)^{j-1} \tilde{I} (|z|, x-y) \eta \int_{\mathbb{R}^n} J( |z|, x-\zeta)\vee J( |z|, \zeta-y) |V(\zeta)| d\zeta.
\end{align*}
Let $\omega=eD_1 D_3 D_4$ and by the definition of $\tilde{K}_V(\zeta)$ we get the desired result.
\end{proof}
To proceed, let
$$ T_j(|z|)f(x)=\int_{\mathbb{R}^n}\tilde{K}_j(|z|,x,y)f(y)dy,  $$
where $f\in L^1.$ Then we have the following lemma.
\begin{lemma}
Let $0< \alpha< 1$ and $\tilde{H}=e^{i\theta}(-\Delta)^{\frac{\alpha}{2}}+e^{i\theta}V$ for $0\le |\theta|<\frac{\pi}{2}$ where $V\in K_\alpha(\mathbb{R}^n).$ Then the following holds for every $|z|>0$
\begin{equation}\label{step2}
  \lim_{N\rightarrow \infty} \| e^{-|z| \tilde{H}} -\sum_{j=0}^{N} (-1)^j T_j(|z|)  \|_{L^1,L^1}=0.
\end{equation}
\end{lemma}
\begin{proof}
Note first that $e^{i\theta}(-\Delta)^{\frac{\alpha}{2}}$ generates an analytic semigroup of angle $\frac{\pi}{2}-|\theta|$ on $L^1(\mathbb{R}^n).$ Since $V \in K_\alpha(\mathbb{R}^n)$, then for each $\varepsilon >0,$ there exists $C_\varepsilon>0$ such that(\cite{ZY})
$$ \| e^{i\theta} V\phi \|_{L^1} \le \varepsilon \| e^{i\theta} (-\Delta)^{\frac{\alpha}{2}} \phi \|+C_\varepsilon \| \phi \|_{L^1} \hspace{1cm} \forall \phi \in \mathcal{L}^{2\alpha,1}(\mathbb{R}^n). $$
Then $\tilde{H}$ generates an analytic semigroup and hence can be represented as for certain  proper path $\Gamma$
$$ e^{-|z| \tilde{H}}=\frac{1}{2\pi i} \int_{\Gamma} e^{\mu |z|} (\mu+ \tilde{H})^{-1} d\mu.$$

Moreover there exist large enough  $\omega>0$ and $ \varepsilon>0$ such that the following holds for $\mu\in \omega+ \Sigma_{\pi-|\theta|}=\{z: |\arg z|<\pi-|\theta|  \}$
\begin{align*}
   & \| e^{i\theta}V( \mu+ e^{i\theta} (-\Delta)^{\frac{\alpha}{2}} )^{-1} \|_{L^1, L^1} \\
   & \le \varepsilon \| e^{i\theta} (-\Delta)^{\frac{\alpha}{2}} (\mu+ e^{i\theta} (-\Delta)^{\frac{\alpha}{2}})^{-1} \|_{L^1,L^1}+C_\varepsilon \| (\mu+e^{i\theta}(-\Delta)^{\frac{\alpha}{2}})^{-1} \|_{L^1, L^1}\\
   &<\frac{1}{2}.
\end{align*}
As a result, for $\mu\in \omega+ \Sigma_{\pi-|\theta|}$ we have
$$ (\mu+\tilde{H})^{-1}=\sum_{j=0}^{\infty} (-1)^j (\mu+ e^{i\theta} (-\Delta)^{\frac{\alpha}{2}})^{-1} (e^{i\theta}V(\mu+e^{i\theta} (-\Delta)^{\frac{\alpha}{2}})^{-1} )^j ,$$
and
$$ (\mu+\tilde{H})^{-1}- \sum_{j=0}^{N} (-1)^j (\mu+ e^{i\theta} (-\Delta)^{\frac{\alpha}{2}})^{-1} (e^{i\theta}V(\mu+e^{i\theta} (-\Delta)^{\frac{\alpha}{2}})^{-1} )^j =(-1)^{N+1} r_N(\mu), $$
where $r_N(\mu)=(\mu+ e^{i\theta} (-\Delta)^{\frac{\alpha}{2}})^{-1} (e^{i\theta}V(\mu+e^{i\theta} (-\Delta)^{\frac{\alpha}{2}})^{-1} )^N e^{i\theta}V (\mu+\tilde{H})^{-1} .$

Then $r_N(\mu)$ is an analytic function satisfying
$$ \sup \{ \| (\mu-\omega) r_N(\mu) \|_{L^1, L^1} : \quad \mu \in \omega+ \Sigma_{\pi- |\theta|}  \}   \le C 2^{-N}.  $$
It follows that
$$ \| \int_\Gamma e^{\mu t} r_N(\mu) d\mu \|_{L^1, L^1} \le C 2^{-N} e^{\omega |z|} \rightarrow 0 \hspace{0.5cm} \text{as} ~~~~~  N\rightarrow \infty , $$
where $\Gamma= \Gamma_0+\Gamma_{\pm},$ $\Gamma_0=\{ \mu :~~ \mu = w+ \delta e^{i\psi}, |\psi| \le \theta_1+\frac{\pi}{2} \}$ and $ \Gamma_{\pm}=\{ \mu :~~ \mu = w+ r e^{\pm i(\theta_1+\frac{\pi}{2})}, r\ge \delta \} $ $(0<\theta_1<\theta, \delta>0).$
Then we obtain
$$ \sum_{j=0}^{N} (-1)^j \int_\Gamma e^{\mu |z|} (\mu+ e^{i\theta} (-\Delta)^{\frac{\alpha}{2}})^{-1} (e^{i\theta}V(\mu+e^{i\theta} (-\Delta)^{\frac{\alpha}{2}})^{-1} )^j d\mu \rightarrow \int_\Gamma e^{\mu |z|} (\mu+ \tilde{H})^{-1} d\mu , $$
in operator norm on $L^1(\mathbb{R}^n)$ as $N$ goes to infinity. By the uniqueness of Laplace transforms, it is sufficient to prove  $(\mu+ e^{i\theta} (-\Delta)^{\frac{\alpha}{2}})^{-1} (e^{i\theta}V(\mu+e^{i\theta} (-\Delta)^{\frac{\alpha}{2}})^{-1} )^j$ and the Laplace transform of  $T_j(|z|)$ coincide.

For $j\ge1$, let
$$ R_j(\mu, x, y)=\int_{\mathbb{R}^n} R_{j-1}(\mu, x, y)e^{i\theta}R_0(\mu, x, y) dz, $$
where $R_0(\mu, x, y)=R(\mu, x, y)=(\mu-e^{i\theta}(-\Delta)^{\frac{\alpha}{2}})^{-1}.$

To start with, we have
\begin{align*}
  |R_0(\mu, x, y)| & =\int_0^\infty e^{-t\mu} \tilde{K}_0(t,x,y) dt   \\
     & \le \int_0^\infty e^{-t\mu} \tilde{I}(t,x,y) dt\\
     & \le D_3 \eta J(\mu^{-1}, x-y).
\end{align*}
Therefore by induction
$$ |R_j(\mu, x, y)|\le C \tilde{K}_V(\mu^{-1})^{j} \eta J(\mu^{-1}, x-y). $$
It follows that $R_j(\mu, x, y)$ is well defined for each $j$ and is actually the kernel of the operator $(\mu+e^{i\theta}(-\Delta)^{\frac{\alpha}{2}})^{-1}(e^{i\theta}V(\mu+e^{i\theta}(-\Delta)^{\frac{\alpha}{2}})^{-1})^j$. Then we have
\begin{align*}
   & \int_0^\infty e^{-t\mu} \tilde{K}_{j+1}(t,x,y)dt  \\
   & = \int_0^\infty e^{-t\mu} \int_{\mathbb{R}^n} \int_0^t \tilde{K}_j(t-s, x, z) e^{i\theta}V(z) \tilde{K}_0(s, z, y) ds dz dt\\
   & = \int_{\mathbb{R}^n} e^{i\theta}V(z) dz \int_0^\infty e^{-t\mu} \tilde{K}_j(t,x,z)dt \int_0^\infty e^{-s\mu} \tilde{K}_0(s,z,y)ds\\
   & = \int_{\mathbb{R}^n} R_j(\mu, x,z) e^{i\theta}V(z)R(\mu, z, y) dz =R_{j+1}(\mu, x, y).
\end{align*}
We have used the Fubini's Theorem in the second step which is due to the fact
$$ \int_0^\infty e^{-t\mu} |\tilde{K}_j(t,x, y)| dt \le C \tilde{K}_V(\mu^{-1})^{j} \eta J(\mu^{-1}, x-y).  $$
Finally we obtain
\begin{align*}
  \int_0^\infty e^{-t\mu} T_j(t) f(x) dt & = \int_0^\infty e^{-t\mu} \int_{\mathbb{R}^n} \tilde{K}_j(t, x, y) f(y) dy dt  \\
   & =\int_{\mathbb{R}^n} R_j(\mu,x,y)f(y)dy \\
   &=(\mu+e^{i\theta}(-\Delta)^{\frac{\alpha}{2}})^{-1}(e^{i\theta}V(\mu+e^{i\theta}(-\Delta)^{\frac{\alpha}{2}})^{-1})^jf(x).
\end{align*}
We have used the fact in the second step
\begin{align*}
  \int_{\mathbb{R}^n} \int_0^\infty |e^{-t\mu}\tilde{K}_j(t,x,y)|dtdy & \le C \tilde{K}_V(\mu^{-1})^j \eta \int_{\mathbb{R}^n} J( \mu^{-1}, x-y)dy \\
   & \le C \tilde{K}_V(\mu^{-1})^j \frac{\eta}{\mu}.
\end{align*}
Thus we have proved the lemma.
\end{proof}
Now we are ready to prove Theorem 1.3 for $0<\alpha<1$.
\begin{proof}[Proof of (1) of Theorem 1.3]
For $0<\varepsilon<1$, set $$ V^\varepsilon =\sup \{ \sigma \le 1 : ~~  t\in (0, \sigma), \omega \tilde{K}_V(t) \le \varepsilon   \} .$$
Denote $$ T(|z|)f(x)= \int_{\mathbb{R}^n} \tilde{K}(|z|,x,y) f(y) dy, $$ where $\tilde{K}(|z|,x,y)=\sum_{j\ge0} \tilde{K}_j(|z|,x,y)$. Thus by Lemma 3.2, we have
$$ |\tilde{K}(|z|, x, y) |\le \sum_{j=0}^{\infty} D_2 (\omega \tilde{K}_V(|z|) )^j \tilde{I} (|z|, x-y) \le \frac{D_2}{1-\varepsilon} \tilde{I} (|z|, x-y), $$
and for $0<|z|<V^\varepsilon$
$$ \lim_{N\rightarrow \infty} \| T(|z|)- \sum_0^N (-1)^j T_j(|z|) \|_{L^1, L^1}=0. $$
Then by Lemma 3.3 we conclude that $\tilde{K}(|z|, x, y)$ coincides with $K(z,x,y)$ which is the kernel of $e^{-z((-\Delta)^{\frac{\alpha}{2}} + V  )}$ for $0<|z|<V^\varepsilon.$ Now we will pass the estimates above to the general case $|z|>0$. Then for $|z| \in (V^\varepsilon, 2V^\varepsilon)$ we have by semigroup property
$$ K (z,x,y)= \int_{\mathbb{R}^n} \tilde{K}\left(\frac{|z|}{2},x, \zeta \right) \tilde{K}\left(\frac{|z|}{2}, \zeta ,y \right) d\zeta.   $$
It follows that
\begin{align*}
  |K(z,x,y) | & \le \left( \frac{D_1}{1-\varepsilon} \right)^2 \tilde{I} (|z|, x-y) \int_{\mathbb{R}^n} \left|\tilde{K}\left(\frac{|z|}{2},x,\zeta \right)\right|+ \left|\tilde{K}\left(\frac{|z|}{2},\zeta ,y \right) \right| d\zeta   \\
   & \le 2D_3D_5 \left( \frac{D_1}{1-\varepsilon} \right)^2 \tilde{I} (|z|, x-y),
\end{align*}
where $D_5= \int_{\mathbb{R}^n} \tilde{I}(|z|,x-y) dy $ is independent of $|z|$ and $x$.

By inductive argument, we have for $|z| \in (2^{n-1}V^\varepsilon, 2^n V^\varepsilon )$
$$ | K(z,x,y) | \le \frac{1}{2D_3D_5} \left(  \frac{2D_1 D_3 D_5}{1-\varepsilon} \right)^{2^n} \tilde{I} (|z|, x-y).  $$
Let $\mu_{\varepsilon, V} = \frac{2 \ln A}{V^\varepsilon}$ where $A= \frac{2D_1 D_3 D_5}{1-\varepsilon}$ and we obtain
$$ | K(z,x,y) | \le \frac{1}{2D_3D_5} e^{\mu_{\varepsilon, V}|z|}  \tilde{I} (|z|, x-y). $$
Thus we have completed the proof.
\end{proof}

\section{Appendix}
In this section, we gather some facts about the Bessel functions as well as the auxiliary functions which are frequently used.

Denote by $J_\nu(z)$ the bessel function for $\Re \nu>-\frac{1}{2}$ and $|\arg z|<\pi$  which can be defined by (\cite[p.211]{NU})
\begin{equation}\label{Bessel def}
   J_\nu(z)=\sum_{k\ge 0} a_k z^{\nu+2k}, \hspace{1cm} \text{with}~~~~ a_k=\frac{(-1)^k 2^{1-2k-\frac{n}{2}}}{k! \Gamma(k+\frac{n}{2})}.
\end{equation}
Moreover, we have the  asymptotic development of $J_\nu(z)$ as $z\rightarrow \infty$ (\cite[p.209]{NU})
\begin{equation}\label{Bessel pro2}
  J_\nu(z)= \frac{1}{2}[H^{(1)}_\nu(z)+H^{(2)}_\nu(z)] \sim z^{-\frac{1}{2}}e^{iz} \sum_{k\ge0} b_k z^{-k} +  z^{-\frac{1}{2}}e^{-iz} \sum_{k\ge0} b'_k z^{-k},
\end{equation}
where $b_k=(\frac{1}{2\pi})^{\frac{1}{2}}e^{-i(\frac{\pi \nu}{2}+\frac{\pi}{4})} \frac{i^k\Gamma(\nu+\frac{1}{2} + k)}{ 2^k k! \Gamma(\nu+\frac{1}{2} - k)}$ and
$b'_k=(\frac{1}{2\pi})^{\frac{1}{2}}e^{i(\frac{\pi \nu}{2}+\frac{\pi}{4})} \frac{(-i)^k\Gamma(\nu+\frac{1}{2} + k)}{ 2^k k! \Gamma(\nu+\frac{1}{2} - k)}$. The above expansion holds in the sense that
$$ \sum_{k\ge N}b_k z^{-k} \triangleq \frac{1}{2} z^{\frac{1}{2}}e^{-iz} H^{(1)}_\nu(z) - \sum_{k=0}^{N-1} b_k z^{-k} = O(z^{-N})\hspace{0.3cm} \text{as} \hspace{0.3cm} |z|\rightarrow \infty; $$
$$ \sum_{k\ge N}b'_k z^{-k} \triangleq \frac{1}{2} z^{\frac{1}{2}}e^{iz} H^{(2)}_\nu(z) - \sum_{k=0}^{N-1} b'_k z^{-k} = O(z^{-N})\hspace{0.3cm} \text{as} \hspace{0.3cm} |z|\rightarrow \infty. $$

In our proof, the following properties of the auxiliary functions have been used.
\begin{lemma}
Set $h(s)=(-e^{i\theta}\alpha s^{\alpha-1}+i )^{-1}$ for $s, \alpha>0, 0<\omega\le \theta<\frac{\pi}{2}$. Then for nonnegative integer $\gamma $ there exists constant $C_\gamma>0$ such that
$$ |h^{(\gamma)}(s)|\le C_{\gamma} s^{-\gamma}, \hspace{0.5cm} \forall ~~ 0< s < \delta s_0,~~\text{and} ~~  s>\frac{s_0}{\delta},  $$
where $0<\delta<1$ and $s_0=(\alpha \sin \theta)^{\frac{1}{1-\alpha}}$.
\end{lemma}
\begin{proof}
It is direct to check that
$$ |h(s)|\le |\alpha \sin \theta s^{\alpha -1} -1|^{-1} \le C_0    $$
for each $ 0< s < \delta s_0 $,  $  s>\frac{s_0}{\delta}$ and $\alpha>0, \alpha\neq1$. Since $h'(s)= e^{i\theta}\alpha(\alpha-1)s^{\alpha-2}h^2(s)$, we obtain for $\alpha>0, \alpha\neq1$
$$ |h'(s)| \le \alpha|\alpha-1|  |s^{-1} h(s)| |s^{\alpha-1}h(s)|\le C'_1  s^{-1}|h(s)|\le C_1 s^{-1}  $$
where $ 0< s < \delta s_0 $ or  $  s>\frac{s_0}{\delta}$. Then for $\gamma\ge 2$ we have
$$ h^{(\gamma)}(s)=(h')^{(\gamma-1)}(s)=\sum_{k_1, k_2, k_3} c_{k_1, k_2, k_3} s^{\alpha-2-k_1}h^{(k_2)}(s) h^{(k_3)}(s)  $$
where $k_1, k_2, k_3 \ge 0$ and $k_1+ k_2 + k_3=\gamma-1$. Since we have proved $|h'(s)|\le C'_1 s^{-1} |h(s)|$,  by induction, we have
$$ |h^{(\gamma)}(s)|\le C'_{\gamma}s^{-\gamma} |h(s)| $$
for $ 0< s < \delta s_0 $,  $  s>\frac{s_0}{\delta}$ and $\alpha>0, \alpha\neq1$. Thus the result follows.
\end{proof}
Specifically when $\gamma=1, 0<\alpha<1 $,  we also have for $0<s < \delta s_0$
$$ |h'(s)|\le C s^{-\alpha} s^{2(\alpha-1)} |h^2(s)| \le C' s^{-\alpha}.  $$

%\section*{Acknowledgements}
%The authors thank the Professor Alexander Grigor'yan for helpful suggestions. Shiliang Zhao was funded by the National Natural Science Foundation of China under Grant No.11901407, No.11971327 and the Fundamental Research Funds for the Central Universities No.2021SCU12105. Quan Zheng was funded by the National Natural Science Foundation of China under Grant No.12171178.

\bibliographystyle{amsplain}

\begin{thebibliography}{10}

\bibitem{BG} R.M. Blumenthal,  R.K.  Getoor,  Some theorems on stable processes. Trans. Amer. Math. Soc. 95, 263-273 1960.

\bibitem{BD} G. Barbatis, E. Davies  Sharp bounds on heat kernels of higher order uniformly elliptic operators, J.Operator Theory, 36(1996), 179-198.

\bibitem{BJ} K. Bogdan, T. Jackubowski,  Estimates of heat kernel of fractional Laplacian perturbed by gradient opertors, Comm. Math. Phys. 271(2007), 179-198.

\bibitem{CCO} G. Carron, T. Coulhon,  and E.-M. Ouhabaz,  Gaussian estimates and $L^p$-boundedness of Riesz means. J. Evol. Equ. 2(2002), 299-317.

\bibitem{CS2007} T. Coulhon,  A. Sikora,  Gaussian heat kernel upper bounds via the Phragm\'en-Linde\"of theorem. Proc. Lond. Math. Soc. 96(2007), 507-544.

\bibitem{CS} L. Caffarelli, L. Silvestre, An extension problem related to the fractional Laplacian, Commun. Partial Differ. Equ. 32(2007), 1245-1260

\bibitem{CHKL}  Y.  Cho,  G. Hwang,  S. Kwon, and S. Lee,  Well-posedness and ill-posedness for the cubic fractional Schr\"odinger equations, Discrete Contin. Dyn. Syst. 35, 2863-2880, 2015.

\bibitem{Davies95} E.B. Davies, Uniformly elliptic operators with measurable coefficients, J. Funct. Anal. 132(1995), 141-169.

\bibitem{Davies89} E.B. Davies, Heat Kernels and Spectral Theory, Cambridge Univ. Press, 1989.

\bibitem{DDY} Q. Deng,  Y. Ding, and X. Yao, Gaussian bounds for higher-order elliptic differential operators with Kato type potentials, J.Funct. Anal. 266 (2014) 5377-5397.

\bibitem{DZF} Z. Duan,  Q.  Zheng, and J.  Feng, Long range scattering for higher order Schr\"odinger operators, J.Differential Equations, 254 (2013), 3329-3351.

%\bibitem{E}  Erdelyi, A., Higher transcendental functions,  vol. II, New York, Bateman  Manuscript Project, 1953.

\bibitem{GH}  B. Guo,  Z. Huo,  Global well-posedness for the fractional nonlinear Schr\"odinger equation, Comm. Partial Differential Equations, 36 (2010), 247-255.

\bibitem{Gri09} A. Grigor'yan,  Heat Kernel and Analysis on Manifolds. AMS/IP Studies in Advanced Mathematics 47. Amer. Math. Soc., Providence, RI.

%\bibitem{GR} I.S. Gradshteyn,  I. M.  Ryzhik,  Table of Integrals, Series, and Products. Edited by A. Jeffrey and D. Zwillinger. Academic Press, New York, 7th edition, 2007.

\bibitem{HNS} N. Hayashi,  E.I.  Naumkin, and  I.A. Shishmarev,   Asymptotics for Dissipative Nonlinear Equations, in: Lecture Notes in Mathematics, vol. 1884,
Springer-Verlag, Berlin, Heidelberg, 2006.

\bibitem{HS} Y. Hong,  Y. Sire,  On fractional Schr\"odinger equations in Sobolev spaces, Commun. Pure Appl. Anal.,  14 (2015), 2265-2282.

\bibitem{HHZ} T. Huang, S. Huang, and Zheng, Q., Inhomogeneous oscillatory integrals and global smoothing efffects for dispersive equations, J. Differential Equations., 263(2017), 8606-8629.

\bibitem{HWZD} S. Huang, M.  Wang, Q. Zheng,  and Z. Duan, $L^p$  estimates for fractional Schr\"odinger operators with Kato class potentials, J. Differential Equations, 265(2018), 4181-4212.

\bibitem{HYZ} S. Huang,   X. Yao, and  Q. Zheng,  Remarks on Lp-limiting absorption principle of Schr\"odinger operators and applications to spectral multiplier theorems, Forum Math.,  30 (2018), 43-55.

\bibitem{K} A.N.  Kochubei, Parabolic pseudo-differential equations, hypersingular integrals, and Markov processes, Izv. AN SSSR, Ser. Mat. 52 (5)
(1988) 909-934 (in Russian). English translation: Math. USSR Izv., 33 (2) (1989) 233-259.

\bibitem{L2002} N. Laskin,  Fractional Schr\"odinger equation, Phys. Rev. E  66, 05618, 2002.

\bibitem{M}  A. Miyachi,  On some singular Fourier multipliers, J. Fac. Sci. Univ. Tokyo. 28,  267-315, 1981.

\bibitem{MYZ} C. Miao, B. Yuan, and  B. Zhang, Well-posedness of the Cauchy problem for the fractional power dissipative equations, Nonlinear Anal. 68, 461-484, 2008.



\bibitem{NU} A.F. Nikiforov, V.B. Uvarov, Special Functions of Mathematical Physics: A Unified Introduction with Applications. Translated from the Russian by Ralph P. Boas. Birkhauser Verlag, Basel (1988).

\bibitem{Ouh05} E.M. Ouhabaz, Analysis of Heat Equations on Domains, London Math. Soc. Monographs, 31, Princeton Univ. Press, Princeton, NJ, 2005.

\bibitem{S} E.M. Stein, Harmonic analysis : real-variable methods, orthogonality, and oscillatory integrals. Princeton University Press, 2006.

\bibitem{W} S. Wainger, Special trigonometric series in k-dimensions. Memoirs Amer. Math. Soc., 59, 1965.

\bibitem{ZY} Q. Zheng, X. Yao, Higher order Kato class potentials for schr\"odinger operators. Bull. Lond. Math. Soc., 41 293-301 (2009).
\end{thebibliography}

\end{document}